\documentclass[11pt]{amsart}
\usepackage{amsmath}%
\usepackage{amsfonts}%
\usepackage{amssymb}%
\usepackage{graphicx}
\usepackage{enumerate}

\usepackage[margin=2.6cm]{geometry}
\newtheorem{theorem}{Theorem}
\newtheorem{lemma}[theorem]{Lemma}

\newtheorem{corollary}[theorem]{Corollary}
\newtheorem{proposition}[theorem]{Proposition}

\newtheorem{conjecture}[theorem]{Conjecture}
\newtheorem{definition}[theorem]{Definition}

\newtheorem{question}[theorem]{Question}

\def\COMMENT#1{}
\def\TASK#1{}

\def\C{\mathcal{C}}
\def\S{\mathcal{S}}

\def\cR{\mathcal{R}}

\def\d{\delta}
\def\eps{\varepsilon}

\def\a{\alpha}
\def\b{\beta}
\def\g{\gamma}

\def\l23{\log_2\!{3}\,}

\def\Tu{\text{Tu}}

\numberwithin{theorem}{section}
\numberwithin{equation}{section}

\def\noproof{{\unskip\nobreak\hfill\penalty50\hskip2em\hbox{}\nobreak\hfill%
       $\square$\parfillskip=0pt\finalhyphendemerits=0\par}\goodbreak}
\def\endproof{\noproof\bigskip}

\include{epsf}

\title[On the structure of digraphs with forbidden tournaments or cycles]{On the structure of oriented graphs and digraphs with forbidden tournaments or cycles}
\author{Daniela K\"uhn, Deryk Osthus, Timothy Townsend, Yi Zhao}
\thanks{The research leading to these results was partially supported by the European Research Council
under the European Union's Seventh Framework Programme (FP/2007--2013) / ERC Grant
Agreements no. 258345 (D.~K\"uhn) and 306349 (D.~Osthus). Yi Zhao was partially supported by NSA grant H98230-12-1-0283 and NSF Grant DMS-1400073.
}

\begin{document}

\begin{abstract}
Motivated by his work on the classification of countable homogeneous oriented graphs, Cherlin asked about the typical structure of oriented graphs {\rm (i)} without a transitive triangle, or {\rm (ii)} without an oriented triangle. We give an answer to these questions (which is not quite the predicted one). Our approach is based on the recent `hypergraph containers' method, developed independently by Saxton and Thomason as well as by Balogh, Morris and Samotij. Moreover, our results generalise to forbidden transitive tournaments and forbidden oriented cycles of any order, and also apply to digraphs. Along the way we prove several stability results for extremal digraph problems, which we believe are of independent interest.
\end{abstract}

\date{\today}

\maketitle

\section{Introduction}\label{Section: Introduction}

\subsection{$H$-free graphs}

Given a fixed graph $H$, a graph is called $H$-\emph{free} if it does not contain $H$ as a (not necessarily induced) subgraph. In 1976 Erd\H{o}s, Kleitman and Rothschild~\cite{EKR} asymptotically determined the logarithm of the number of $K_k$-free graphs on $n$ vertices, for every $k\geq 3$. This was strengthened by Kolaitis, Pr\"omel and Rothschild~\cite{KPR}, who showed that almost all $K_k$-free graphs are $(k-1)$-partite, for every $k\geq 3$ (the case $k=3$ of this was already proved\COMMENT{Yes, I checked and this was indeed proved in the same paper as the result mentioned above} in~\cite{EKR}). This was one of the starting points for a vast body of work concerning the number and structure of $H$-free graphs on $n$ vertices (see, e.g.~\cite{BBS1, BBS2, BBS3, BMSW, BaSa, EFR, KPR, OPT, PrSt}). The strongest of these results essentially state that for a large class of graphs $\mathcal{H}$, and any $H\in \mathcal{H}$, almost all $H$-free graphs have a similar structure to that of the extremal $H$-free graph. More recently, some related results have been proved for hypergraphs (see, e.g.~\cite{BaMu, PeSc}).

However, the corresponding questions for digraphs and oriented graphs are almost all wide open, and are the subject of this paper. Until now the only results of the above type for oriented graphs were proved by Balogh, Bollob\'as and Morris~\cite{BBM1, BBM2} who classified the possible `growth speeds' of oriented graphs with a given property. Moreover Robinson~\cite{Rob1, Rob2}, and independently Stanley~\cite{Stan}, counted the number of acyclic digraphs. A related problem was considered by Alon and Yuster~\cite{AYoriented}, who determined $\max_G D(G, T)$ over all $n$-vertex graphs $G$ for sufficiently large $n$, where $T$ is a fixed tournament and $D(G, T)$ denotes the number of $T$-free orientations of $G$ (note that $\sum_G D(G, T)$ is the number of $T$-free oriented graphs on $n$ vertices).
%determined the maximum possible number of $T$-free orientations an $n$-vertex graph can have, where $T$ is a fixed tournament.

\subsection{Oriented graphs and digraphs with forbidden tournaments or cycles}

A \emph{digraph} is a pair $(V,E)$ where $V$ is a set of vertices and $E$ is a set of ordered pairs of distinct vertices in $V$ (note that this means that in a digraph we do not allow loops or multiple edges in the same direction). An \emph{oriented graph} is a digraph with at most one edge between two vertices, so may be considered as an orientation of a simple undirected graph. A \emph{tournament} is an orientation of a complete graph. We denote a transitive tournament on $k$ vertices by $T_k$, and a directed cycle on $k$ vertices by $C_k$. We only consider labelled graphs and digraphs.

Given a class of graphs $\mathcal{A}$, we let $\mathcal{A}_n$ denote the set of all graphs in $\mathcal{A}$ that have precisely $n$ vertices, and we say that \emph{almost all graphs in} $\mathcal{A}$ \emph{have property} $\mathcal{B}$ if
$$\lim\limits_{ n\to \infty }\frac{|\{G\in \mathcal{A}_n: G \text{ has property } \mathcal{B}\}|}{|\mathcal{A}_n|} = 1.$$
Clearly any transitive tournament is $C_k$-free for any $k$, and any bipartite digraph is $T_3$-free. In 1998 Cherlin \cite{Cher} gave a classification of countable homogeneous oriented graphs. He remarked that `the striking work of~\cite{KPR} does not appear to go over to the directed case' and made the following conjectures.\footnote{Note that oriented graphs are referred to as digraphs in~\cite{Cher}.}
\begin{conjecture}[Cherlin]\label{Cherlin Conj}
\label{conj1}
~
\begin{enumerate}[{\rm (i)}]
\item Almost all $T_3$-free oriented graphs are tripartite.
\item Almost all $C_3$-free oriented graphs are acyclic, i.e. they are subgraphs of transitive tournaments.
\end{enumerate}
\end{conjecture}

Our first main result not only verifies part {\rm (i)} of this conjecture, but shows for all $k\geq 2$ that almost all $T_{k+1}$-free oriented graphs are $k$-partite. Note that in particular this shows that in fact almost all $T_3$-free oriented graphs are actually even bipartite. We also prove the analogous result for digraphs.

\begin{theorem} \label{thm:main1}
Let $k\in \mathbb{N}$ with $k\geq 2$. Then the following hold.
\begin{enumerate}[{\rm (i)}]
\item Almost all $T_{k+1}$-free oriented graphs are $k$-partite.
\item Almost all $T_{k+1}$-free digraphs are $k$-partite.
\end{enumerate}
\end{theorem}

Theorem~\ref{thm:main1} can be viewed as a directed version of the theorem of Kolaitis, Pr\"omel and Rothschild~\cite{KPR} mentioned earlier. Note also that {\rm (i)} means that the typical structure of a $T_3$-free oriented graph is \emph{not} close to that of the extremal $T_3$-free oriented graph: it is easy to see that the latter is the blow up of a directed triangle\COMMENT{By Tur\'an's theorem the unique extremal (undirected) graph for $K_4$ is the complete balanced tripartite graph, which is the underlying undirected graph of the blow up of a directed triangle. So if an oriented graph has more edges than that of the blow up of a directed triangle then its underlying undirected graph contains a copy of $K_4$. But every orientation of $K_4$ contains a copy of $T_3$. So indeed the blow up of a directed triangle is the extremal $T_3$-free oriented graph (uniqueness follows since the underlying undirected graph of any extremal $T_3$-free oriented graph is the complete balanced tripartite graph, by the uniqueness part of Tur\'an's theorem, and clearly the only $T_3$-free orientation of the complete balanced tripartite graph is the blow up of a directed triangle).} (this fact was probably the motivation for Conjecture~\ref{Cherlin Conj}{\rm (i)}).

Our next main result shows in particular that part {\rm (ii)} of Conjecture~\ref{Cherlin Conj} is in fact false. We actually show something stronger, namely that for all $k\geq 3$ and for almost all $C_k$-free oriented graphs on $n$ vertices, the number of edges we must change in order to get an acyclic oriented graph is $\Omega(n/\log n)$. We also prove an analogous version of this result for digraphs. However, Conjecture~\ref{Cherlin Conj}{\rm (ii)} is not too far from being true, as we prove also that almost all $C_k$-free oriented graphs are close to acyclic, in the sense that we only need to change sub-quadratically many edges in order to obtain an acyclic oriented graph. In the case when $k$ is even we prove an analogous version of this result for digraphs too. We also obtain a (less restrictive) structural result for odd $k$.

In order to state the theorem precisely we need to introduce a little terminology. Given a labelled digraph or oriented graph $G$ with vertex labels $1,\dots, n$ and an ordering $\sigma: [n]\rightarrow [n]$, a \emph{backwards edge} in $G$ with respect to this ordering is any edge directed from a vertex labelled $i$ to a vertex labelled $j$, where $\sigma(i)>\sigma(j)$. A \emph{transitive-optimal ordering} of $G$ is any ordering of $V(G)$ that minimises the number of backwards edges in $G$ with respect to the ordering. We say that a directed graph is a \emph{transitive-bipartite blow up} if it can be obtained from a transitive tournament by replacing some of its vertices by complete balanced bipartite digraphs. More formally, a directed graph $G=(V,E)$ is a transitive-bipartite blow up if $V$ admits a partition $A_1,\dots, A_t$ such that for all $i,j\in [t]$ with $i<j$, the graph induced on $G$ by $A_i$ is either a single vertex or a complete balanced bipartite digraph (with edges in both directions) and the edges in $E$ between $A_i$ and $A_j$ are precisely those edges directed from $A_i$ to $A_j$ (and no others).

\begin{theorem}\label{C_k free main theorem}
Let $k, n\in \mathbb{N}$ with $k\geq 3$. There exists $c>0$ such that for every $\alpha>0$ the following hold.
\begin{enumerate}[{\rm (i)}]
\item Almost all $C_k$-free oriented graphs on $n$ vertices have between $c n/\log n$ and $\alpha n^2$ backwards edges in a transitive-optimal ordering.
\item Almost all $C_k$-free digraphs on $n$ vertices have at least $c n/\log n$ backwards edges in a transitive-optimal ordering. Moreover,
\begin{enumerate}[{\rm (a)}]
\item if $k$ is even then almost all $C_k$-free digraphs on $n$ vertices have at most $\alpha n^2$ backwards edges in a transitive-optimal ordering,
\item if $k$ is odd then almost all $C_k$-free digraphs on $n$ vertices can be made into a subgraph of a transitive-bipartite blow up by changing at most $\alpha n^2$ edges.
\end{enumerate}
\end{enumerate}
\end{theorem}

We believe that in fact almost all $C_k$-free oriented graphs have linearly many backwards edges, and that an analogous result holds for $C_k$-free digraphs in the case when $k$ is even.

\begin{conjecture}
Let $k, n\in \mathbb{N}$ with $k\geq 3$. Then the following hold.
\begin{enumerate}[{\rm (i)}]
\item Almost all $C_k$-free oriented graphs on $n$ vertices have $\Theta (n)$ backwards edges in a transitive-optimal ordering.
\item If $k$ is even then almost all $C_k$-free digraphs on $n$ vertices have $\Theta (n)$ backwards edges in a transitive-optimal ordering.
\end{enumerate}
\end{conjecture}
It is not clear to us what to expect in the case when $k$ is odd\COMMENT{We do NOT have that almost all $C_k$-free digraphs on $n$ vertices are bipartite, if $k$ is odd. This is since the number of acyclic digraphs is approximately
$$n!\frac{2^{\binom{n}{2}}}{(0.474...)(1.488)^n}\geq \frac{n!2^{\binom{n}{2}}}{2^n}\geq n!2^{\frac{n^2}{2}-2n}\geq \left( \frac{n}{e} \right)^n 2^{\frac{n^2}{2}-2n},$$
whereas the number of bipartite digraphs is at most $2^n4^{n^2/4} = 2^{n^2/2}2^n$.
Additional comment: A similar argument as in Lemma~\ref{linearly many backwards} where flippable 4-sets are replaced by double edges (in the $C_3$-free case) shows that we must have at least a linear number of backwards edges}.

\begin{question}
Suppose that $k$ is odd and $\alpha >0$. Do almost all $C_k$-free digraphs on $n$ vertices have at most $\alpha n^2$ backwards edges in a transitive-optimal ordering?
\end{question}

An undirected version of Theorem~\ref{C_k free main theorem} for forbidden odd cycles was proved by Lamken and Rothschild~\cite{LaRo}, who showed that for odd $k$, almost all $C_k$-free graphs are bipartite. (So the situation for oriented graphs is very different from the undirected one.) For even $k$ the undirected problem is far more difficult. Despite major recent progress by Morris and Saxton~\cite{MoSa} the problem of counting the number of $C_k$-free graphs is still open for even $k$.

We remark that in Theorem~\ref{thm:main1} we actually get exponential bounds on the proportion of $T_{k+1}$-free oriented graphs and digraphs that are not $k$-partite. We also get similar exponential bounds in Theorem~\ref{C_k free main theorem}.

\subsection{Sketch of proofs}\label{sec:sketch}

A key tool in our proofs is a recent and very powerful result of Saxton and Thomason~\cite{SaTh} as well as Balogh, Morris and Samotij~\cite{BMS}, which gives an upper bound on the number of independent sets in certain hypergraphs. 
Briefly, the result states that under suitable conditions on a uniform hypergraph $G$, there is a small collection $\mathcal{C}$ of small subsets (known as containers) of $V(G)$ such that every independent set of vertices in $G$ is a subset of some element of $\mathcal{C}$. We will use the formulation of Saxton and Thomason~\cite{SaTh}. The precise statement of this result (Theorem~\ref{Original Containers}) is deferred until Section~\ref{Section: Digraph containers}. Roughly speaking, the use of hypergraph containers allows us to reduce an asymptotic counting problem to an extremal problem.
It should be noted that the method of hypergraph containers is much more general than Theorem~\ref{Original Containers}.
Saxton and Thomason~\cite{SCFSH} also gave a short proof of a somewhat weaker version of Theorem~\ref{Original Containers}, which would still have been
sufficient for our purposes. 
%Related results to that of Saxton and Thomason have been proved independently by Balogh, Morris and Samotij~\cite{BMS}. (These would also have been sufficient for our purposes.) 

Our approach to proving our main results is as follows. Firstly, in Section~\ref{Section: Digraph containers} we use the main result of~\cite{SaTh} to derive a container result which is applicable to digraphs. Then in Section~\ref{Section: Rough structure of $T_k$-free oriented graphs and digraphs} we apply this digraph containers result in a relatively standard way to show that almost all $T_{k+1}$-free oriented graphs, and almost all $T_{k+1}$-free digraphs, are close to $k$-partite (see Lemma~\ref{step 0}). In Section~\ref{Section: Exact structure of $T_k$-free oriented graphs and digraphs} we combine Lemma~\ref{step 0} with an inductive argument to prove the results on the exact structure of typical $T_{k+1}$-free oriented graphs and digraphs given by Theorem~\ref{thm:main1}.

In Section~\ref{Section: Rough structure of $C_k$-free oriented graphs and digraphs} we use the digraph containers result to show that almost all $C_k$-free oriented graphs are close to acyclic, and that the analogous result for digraphs holds in the case when $k$ is even (see Lemma~\ref{quadratically many backwards}{\rm (i), (ii)}). For odd $k$ we show that almost all $C_k$-free digraphs are close to a subgraph of a transitive-bipartite blow up (see Lemma~\ref{quadratically many backwards}{\rm (iii)}). Finally, in Section~\ref{Section: Typical C_k-free oriented graphs and digraphs are not transitive} we complete the proof of Theorem~\ref{C_k free main theorem} by giving a lower bound on the number of backwards edges in $C_k$-free oriented graphs and digraphs, the upper bounds in Theorem~\ref{C_k free main theorem} being given by Lemma~\ref{quadratically many backwards}.

As part of the proofs in Sections~\ref{Section: Rough structure of $T_k$-free oriented graphs and digraphs} and~\ref{Section: Rough structure of $C_k$-free oriented graphs and digraphs} we prove several stability results on digraphs which we believe are of independent interest:
\begin{enumerate}[{\rm (i)}]
\item Suppose $k\in \mathbb{N}$ and $G$ is a $T_{k+1}$-free digraph on $n$ vertices with $e(G)\geq {\rm ex}_{di}(n, T_{k+1}) - o(n^2)$. Then $G$ is close to a complete balanced $k$-partite digraph. (See Lemma~\ref{lem:sta}).
\item Suppose $k\in \mathbb{N}$ with $k\geq 4$ and $k$ even, and suppose $G$ is a $C_k$-free digraph on $n$ vertices with $e(G)\geq {\rm ex}_{di}(n, C_k) - o(n^2)$. Then $G$ is close to a transitive tournament. (See Lemma~\ref{cycle stability digraphs}).
\item Suppose $k\in \mathbb{N}$ with $k\geq 3$ and $k$ odd, and suppose $G$ is a $C_k$-free digraph on $n$ vertices with $e(G)\geq {\rm ex}_{di}(n, C_k) - o(n^2)$. Then $G$ is close to a transitive-bipartite blow up. (See Lemma~\ref{odd cycle stability}).
\end{enumerate}
Here ${\rm ex}_{di}(n, H)$ denotes the maximum number of edges among all $H$-free digraphs on $n$ vertices. The corresponding Tur\'an type results which determine ${\rm ex}_{di}(n, H)$ for $H=T_k$ and $H=C_k$ were proved by Brown and Harary~\cite{BrHa} and H\"aggkvist and Thomassen~\cite{HaTh} respectively. These stability results are used in the proofs of Theorems~\ref{thm:main1}{\rm (ii)} and~\ref{C_k free main theorem}{\rm (ii)}. We actually prove `weighted' generalisations of {\rm (i)} and {\rm (ii)} which can be used to prove the assertions about oriented graphs in Theorems~\ref{thm:main1}{\rm (i)} and~\ref{C_k free main theorem}{\rm (i)}.

Before starting on any of this however, we lay out some notation and set out some useful tools in Section~\ref{Section: Notation and tools}, below.

\section{Notation and tools}\label{Section: Notation and tools}
For a set $X$ we let $X^{(r)}$ denote the set of all (unordered) subsets of $X$ of size $r$. An $r$-\emph{uniform hypergraph}, or $r$-\emph{graph}, is a pair $(V,E)$ where $V$ is a set of vertices and $E\subseteq V^{(r)}$. If $G=(V,E)$ is a graph, digraph, oriented graph or $r$-graph, we let $V(G):= V$, $E(G):= E$, $v(G):= |V(G)|$, and $e(G):=|E(G)|$. For a digraph $G=(V,E)$ define $\Delta^0(G)$ as the maximum of $d^+(v)$ and $d^-(v)$ among all $v \in V$. We write $uv$ for the edge directed from $u$ to $v$. For a vertex $v\in V$, define the out-neighbourhood of $v$ in $G$ to be $N^+_G(v):= \{u\in V: vu\in E\}$, and similarly define the in-neighbourhood of $v$ in $G$ to be $N^-_G(v):= \{u\in V: uv\in E\}$. Given a set $U\subseteq V$, we sometimes also write $N^+_U(v):=N^+_G(v)\cap U$ and define $N^-_U(v)$ similarly. For disjoint subsets $U, U'\subseteq V$ we let $G[U, U']$ denote the subdigraph of $G$ with vertex set $U\cup U'$ whose edge set consists of all edges between $U$ and $U'$ in $G$ (in both directions). We let $e(U, U'):= e(G[U, U'])$. If $Q$ is a $k$-partition of $[n]$ with partition classes $V_1\dots, V_k$, and $G$ is a (di)graph or oriented graph with vertex set $[n]$, we say that $Q$ is a $k$-\emph{partition of} $G$ if for every $i\in [k]$ we have that $E(G)$ contains no edges $uv$ with $u,v\in V_i$. We assume $k$-partitions to be unordered unless otherwise stated. For two digraphs $G$ and $G'$ on vertex set $[n]$, we write $G= G' \pm \eps n^2$ if $G$ can be obtained from $G'$ by changing (i.e.~adding, deleting, or changing the orientation of) at most $\eps n^2$ edges. Given an $r$-graph $H=(V,E)$ and $\sigma \in V^{(d)}$, where $0\leq d\leq r-1$, let $d_H(\sigma):=|\{e\in E : \sigma \subseteq e\}|$ be the \textit{degree} of $\sigma$ in $H$. We may simply write $d(\sigma)$ for $d_H(\sigma)$ when it is obvious which $r$-graph $H$ we are working with. The \emph{average vertex degree} of $H$ is defined to be $(1/|V|)\sum_{v\in V} d_H(\{v\})$. Given an oriented graph $H$ and a digraph $H'$,
\begin{itemize}
\item let $f(n, H)$ denote the number of labelled $H$-free oriented graphs on $n$ vertices,
\item let $T(n,k)$ denote the number of labelled $k$-partite oriented graphs on $n$ vertices,
\item let $f^*(n, H')$ denote the number of labelled $H'$-free digraphs on $n$ vertices,
\item let $T^*(n,k)$ denote the number of labelled $k$-partite digraphs on $n$ vertices.
\end{itemize}
In some proofs, given $a,b\in \mathbb{R}$ with $0<a,b<1$, we will use the notation $a\ll b$ to mean that we can find an increasing function $g$ for which all of the conditions in the proof are satisfied whenever $a\leq g(b)$. Throughout the paper we write $\log x$ to mean $\log_2 x$, and we assume all graphs, oriented graphs, and digraphs to be labelled unless otherwise stated. We also assume all large numbers to be integers, so that we may sometimes omit floors and ceilings for the sake of clarity.

We define $H(p):=-p\log p-(1-p)\log (1-p)$, the binary entropy function. The following bound will prove useful to us. For $n\geq 1$ and $0<p<1/2$,
\begin{equation}\label{entropy bound}
\binom{n}{\leq pn}:= \sum\limits_{i=0}^{\left\lfloor pn \right\rfloor} \binom{n}{i} \leq 2^{H(p)n}.
\end{equation}

In a number of our proofs we shall also use the following Chernoff bound.
\begin{theorem}[Chernoff bound]\label{Chernoff Bounds}
Let $X$ have binomial distribution and let $a>0$. Then
$$P(X<\mathbb{E}[X]-a)<\exp \left( -\frac{a^2}{2\mathbb{E}[X]} \right).$$
\end{theorem}

\section{Digraph containers}\label{Section: Digraph containers}

Our main tool is Theorem~\ref{Original Containers} from~\cite{SaTh}. Given a hypergraph $G$ satisfying certain degree conditions, it gives a small set of almost independent sets in $G$ (containers) which together contain all independent sets of $G$. In our applications the vertex set of $G$ will be the edge set of the complete digraph, and the hyperedges will correspond to copies of the forbidden subdigraph. To formulate the degree conditions we need the following definition.

\begin{definition}
Let $G$ be an $r$-graph on $n$ vertices with average vertex degree $d$. Let $\tau >0$. Given $v\in V(G)$ and $1\leq j\leq r, n$, let
$$d^{(j)}(v):= \max \{d(\sigma) : v\in \sigma \subseteq V(G), |\sigma|=j \}.$$
If $d>0$ we define $\delta_j= \delta_j(\tau)$ by the equation
$$\delta_j \tau^{j-1}nd = \sum\limits_{v\in V(G)} d^{(j)}(v).$$
Then the \textit{co-degree function} $\delta(G, \tau)$ is defined by
$$\delta(G, \tau):= 2^{\binom{r}{2}-1} \sum\limits_{j=2}^r 2^{-\binom{j-1}{2}} \delta_j.$$
If $d=0$ we define $\delta(G, \tau):= 0$.
\end{definition}

\begin{theorem}\cite[Corollary 2.7]{SaTh}\label{Original Containers}
\label{thm:ST}
Suppose that $0< \eps < \frac12$ and $\tau \le \frac{1}{144 r^2 ! r}$. Let $G$ be an $r$-graph with vertex set $[n]$ satisfying $\d(G, \tau) \le \frac{\eps}{12 r!}$. Then there exists a constant $c= c(r)$ and a collection $\C$ of subsets of $[n]$ with the following properties.
\begin{enumerate}[{\hspace{0.11cm} \rm (a)}]
\item For every independent set $I$ of $G$ there exists $C\in \C$ such that $I \subseteq C$.
\item $e(G[C]) \le \eps e(G)$ for all $C\in \C$.
\item $\log |\C| \le c \log(\frac{1}{\eps}) n \tau \log(\frac{1}{\tau})$.
\end{enumerate}
\end{theorem}

We will apply Theorem~\ref{thm:ST} to prove Theorem~\ref{thm:contain} below, which is a digraph analogue to \cite[Theorem 1.3]{SaTh}. To state this we need the following definitions. Given a digraph $G=(V,E)$, let $f_1(G)$ be the number of pairs $u, v\in V$ such that exactly one of $uv$ and $vu$ is an edge of $G$, and let $f_2(G)$ be the number of pairs $u, v\in V$ such that both $uv$ and $vu$ are edges of $G$. The following definition of the \emph{weighted size} of $G$ will be crucial in this paper. For $a\in \mathbb{R}$ with $a\geq 1$ we define
\[
e_a(G) := a\cdot f_2(G) + f_1(G).
\]
This definition allows for a unified approach to extremal problems on oriented graphs and digraphs. We will be mainly interested in the cases $a=2$ and $a=\log 3$. The former is useful because each digraph $G$ contains $4^{f_2(G)} 2^{f_1(G)} = 2^{e_2(G)}$ (labelled) subdigraphs, and the latter is useful because each digraph $G$ contains $3^{f_2(G)} 2^{f_1(G)} = 2^{e_{\log 3}(G)}$ (labelled) oriented subgraphs. Given a digraph $H$, define the \emph{weighted Tur\'an number} ${\rm ex}_a(n, H)$ as the maximum $e_a(G)$ among all $H$-free digraphs $G$ on $n$ vertices. (So ${\rm ex}_2(n, H)$ equals ${\rm ex}_{di}(n, H)$ which was defined in Section~\ref{sec:sketch}.) For $A, B\subseteq V$ we will sometimes write $e_a(A,B)$ to denote $e_a(G[A, B])$.

Given an oriented graph $H$ with $e(H)\ge 2$, we let
$$m(H)= \max\limits_{H'\subseteq H, e(H')>1} \frac{e(H')-1}{v(H')-2}.$$

\begin{theorem}\label{thm:contain}
Let $H$ be an oriented graph with $h:= v(H)$ and $e(H)\ge 2$, and let $a\in \mathbb{R}$ with $a\geq 1$. For every $\eps>0$, there exists $c>0$ such that for all sufficiently large $N$, there exists a collection $\C$ of digraphs on vertex set $[N]$ with the following properties.
\begin{enumerate}[{\hspace{0.11cm} \rm (a)}]
\item For every $H$-free digraph $I$ on $[N]$ there exists $G\in \C$ such that $I \subseteq G$.
\item Every digraph $G\in \C$ contains at most $\eps N^h$ copies of $H$, and $e_a(G) \le {\rm ex}_a(N, H) + \eps N^2 $.
\item $\log |\C| \le c N^{2 - 1/m(H)} \log N$.
\end{enumerate}
\end{theorem}

Note that in {\rm (a)}, since $I, G$ are labelled digraphs, $I \subseteq G$ means that $I$ is contained in $G$ in the labelled sense, i.e. the copy of $I$ in $G$ has the same vertex labels as $I$.

The following corollary is a straightforward consequence of Theorem~\ref{thm:contain}. We will not use it elsewhere in the paper, but we include it to illustrate what one can achieve even just with a `direct' application of Theorem~\ref{thm:contain}. 
%It would be interesting to obtain a version of Corollary~\ref{direct corollary} for general forbidden digraphs $H$.

\begin{corollary}\label{direct corollary}
For every oriented graph $H$ with $e(H)\geq 2$, we have $f(n, H) = 2^{ {\rm ex}_{\log 3}(n, H) + o(n^2) }$ and $f^*(n, H) = 2^{ {\rm ex}_{2}(n, H) + o(n^2) }$.
\end{corollary}
\begin{proof}
We only prove the first part here; the proof of the second part is almost identical. Clearly $f(n,H)\geq 2^{{\rm ex}_{\log 3}(n, H)}$. By Theorem~\ref{thm:contain}, for every $\eps>0$ there is a collection $\C$ of digraphs on $[n]$ satisfying properties (a)--(c). We know that every digraph $G\in \C$ contains $2^{e_{\log 3}(G)}$ oriented subgraphs. Since each $H$-free oriented graph is contained in some $G\in \C$, and $|\C|\le 2^{c n^{2 - 1/m(H)} \log n}$,
\[
f(n, H) \le \sum\limits_{G\in \C} 2^{e_{\log 3}(G)} \le 2^{{\rm ex}_{\log 3}(n, H) + \eps n^2 + o(n^2)}.
\]
We are done by letting $\eps\to 0$.
\end{proof}
It is not possible to extend Theorem~\ref{thm:contain} or Corollary~\ref{direct corollary} to all digraphs $H$. For example, let $DK_3$ denote the double triangle (which consists of three vertices and all possible ordered pairs as edges). Assume that $n$ is even and let $A\cup B$ be a balanced partition of $[n]$. Let $\mathcal{G}$ denote the family of all digraphs $G$ on $[n]$ such that $G[A]$ and $G[B]$ are oriented graphs. It is clear that no $G\in \mathcal{G}$ contains a copy of $DK_3$, and $|\mathcal{G}| = 2^{n^2/2} 3^{2\binom{n/2}{2}}$ (there are 2 choices for each of the $n^2/2$ ordered pairs between $A$ and $B$ and there are $3$ choices for each of the $\binom{n/2}{2}$ unordered pairs on either $A$ or $B$). Furthermore, it is easy to see that ${\rm ex}_2(n, DK_3) = \frac{n^2}{2} + 2 \binom{n/2}{2}$ (and the digraphs in $\mathcal{G}$ with the maximum number of edges are extremal digraphs). Thus 
\[
f^*(n, DK_3)\ge |\mathcal{G}| =  2^{\frac{n^2}2} 3^{2\binom{n/2}{2}} \gg 2^{\frac{n^2}{2} + 2 \binom{n/2}{2}} =2^{ {\rm ex}_2(n, DK_3)}. 
\]

The proof of Theorem~\ref{thm:contain} is similar to that of \cite[Theorem 1.3]{SaTh}. We first define the hypergraph $D(N, H)$, which will play the role of $G$ in Theorem~\ref{Original Containers}.

\begin{definition}
Let $H$ be an oriented graph, let $r:= e(H)$ and let $N\in \mathbb{N}$. The $r$-graph $D(N, H)$ has vertex set $U=([N]\times[N])\setminus \{(i,i): i\in [N]\}$, where $B\in U^{(r)}$ is an edge whenever $B$, considered as a digraph with vertices in $[N]$, is isomorphic to $H$.
\end{definition}

We wish to apply Theorem~\ref{thm:ST} to $D(N, H)$. To do this we require an upper bound on $\delta(D(N, H), \tau)$ for some suitable value of $\tau$. We give one in the following lemma, the proof of which is identical to that of \cite[Lemma 9.2]{SaTh} and is therefore omitted here.

\begin{lemma}\label{deltabound}
Let $H$ be an oriented graph with $r:=e(H)\geq 2$, and let $\gamma \leq 1$. For $N$ sufficiently large, $\delta\left( D(N, H), \gamma^{-1}N^{-1/m(H)} \right) \leq r2^{r^2}v(H)!^2 \gamma.$
\end{lemma}

We now state a supersaturation result, which we will use to bound the number of edges in containers. It is the digraph analogue of the well-known supersaturation result of Erd\H{o}s and Simonovits~\cite{ErdSim}. Its proof is almost the same, and is omitted here\COMMENT{See Appendix}.

\begin{lemma}[Supersaturation]\label{supersaturation}
Let $H$ be a digraph on $h$ vertices, and let $a\in \mathbb{R}$ with $a\geq 1$. For any $\eps>0$, there exists $\d>0$ such that the following holds for all sufficiently large $n$. For any digraph $G$ on $n$ vertices, if $G$ contains at most $\d n^h$ copies of $H$, then $e_a(G)\leq {\rm ex}_a(n, H) + \eps n^2$.
\end{lemma}

We may now apply Theorem~\ref{thm:ST} to $D(N, H)$, using Lemmas~\ref{deltabound} and~\ref{supersaturation}, to obtain Theorem~\ref{thm:contain}. The details of this are identical to the proof of Theorem $1.3$ in \cite{SaTh} and are omitted here\COMMENT{The proof is simply to apply Theorem~\ref{thm:ST} to the hypergraph $D(N, H)$ (using lemma~\ref{deltabound} to show that the assumptions of Theorem~\ref{thm:ST} are satisfied) and check that the containers generated satisfy the conditions of Theorem~\ref{thm:contain} (using Lemma~\ref{supersaturation})}.

\section{Rough structure of typical $T_{k+1}$-free oriented graphs and digraphs}\label{Section: Rough structure of $T_k$-free oriented graphs and digraphs}

In this section we prove a stability result for $T_{k+1}$-free digraphs. We apply this (together with Theorem~\ref{thm:contain}) at the end of this section to determine the `rough' structure of typical $T_{k+1}$-free oriented graphs and digraphs.

The \emph{Tur\'an graph} $\Tu_k(n)$ is the largest complete $k$-partite graph on $n$ vertices (thus each vertex class has $\lfloor n/k \rfloor$ or $\lceil n/k \rceil$ vertices). Let $t_k(n) := e(\Tu_k(n))$. Let $DT_k(n)$ be the digraph obtained from $\Tu_k(n)$ by replacing each edge of $\Tu_k(n)$ by two edges of opposite directions. Obviously, for all $k\in \mathbb{N}$, $DT_{k}(n)$ is $T_{k+1}$-free so ${\rm ex}_a(n, T_{k+1}) \ge e_a(DT_{k}(n)) =a \cdot t_{k}(n)$. In Lemma~\ref{lem:turan} below we show that $DT_{k}(n)$ is the unique extremal digraph for $T_{k+1}$. This result is not needed for any of our proofs, but we believe it is of independent interest, in addition to being useful for illustrating the general proof method of Lemma~\ref{lem:sta}. Note that the case $a=2$ of Lemma~\ref{lem:turan} is already due to Brown and Harary~\cite{BrHa}.
\begin{lemma}\label{lem:turan}
Let $a\in \mathbb{R}$ with $3/2< a\leq 2$ and let $k,n\in \mathbb{N}$. Then ${\rm ex}_a(n, T_{k+1}) = a \cdot t_{k}(n)$, and $DT_{k}(n)$ is the unique extremal $T_{k+1}$-free digraph on $n$ vertices.
\end{lemma}

\begin{proof}
Note that $DT_{k}(n)$ is the unique $k$-partite digraph $D$ on $n$ vertices which maximises $e_a(D)$. Moreover, $e_a(DT_{k}(n)) =  a\cdot t_{k}(n)$. Thus, it suffices to show that for every non-$k$-partite $T_{k+1}$-free digraph $G$ on $n$ vertices, there exists a $k$-partite digraph $H$ on the same vertex set such that $e_a(G) < e_a(H)$.

We prove this by induction on $k$. This is trivial for the base case $k=1$, as the only $T_2$-free digraph is the empty graph. Suppose that $k>1$ and that the claim holds for $k-1$. Let $G=(V,E)$ be a non-$k$-partite $T_{k+1}$-free digraph on $n$ vertices. Without loss of generality, suppose that $d^+(x)= \Delta^0(G)$ for some vertex $x\in V$. Let $S:= N^+(x)$ and $T:= V\setminus S$. Since $G$ is $T_{k+1}$-free we have that $G[S]$ is $T_{k}$-free. By induction hypothesis, either {\rm (i)} there is a $(k-1)$-partite digraph $H'$ on $S$ such that $e_a(G[S]) < e_a(H')$, or {\rm (ii)} $G[S]$ is $(k-1)$-partite and hence there is trivially a $(k-1)$-partite digraph $H'$ on $S$ such that $e_a(G[S]) = e_a(H')$. Next we want to replace all the edges inside $T$ with edges between $S$ and $T$ as follows. Suppose $y, z\in T$ with $yz\in E$. Then there are $y', z' \in S$ with $y y'\not\in E$ and $z' z\not\in E$ otherwise $d^+(y)\ge |S| + 1$ or $d^-(z)\ge |S| + 1$, contradicting the assumption $ \Delta^0(G)= d^+(x) = |S|$. We now replace $yz$ with $y y'$ and $z' z$. By the definition of $e_a(\cdot)$, and since $a\leq 2$, the gain of adding the edges $y y'$ and $z' z$ is at least $2(a - 1)$, while the loss of removing $yz$ is at most one. Thus, since $a>3/2$, we have that $e_a(G)$ increases by $2a - 3 > 0$. Note that this procedure does not change the in-degree or out-degree of any vertex in $T$. We repeat this for every edge inside $T$. At the end we obtain a digraph $G'$ with no edge inside $T$. We replace $G'[S] = G[S]$ with the $(k-1)$-partite digraph $H'$ obtained before to obtain a $k$-partite digraph $H$. We now consider the two cases {\rm (i)} and {\rm (ii)} discussed previously. If $e_a(G[S]) < e_a(H')$ then it is clear that $e_a(G) < e_a(H)$, and we are done. Otherwise, $G[S]$ is $(k-1)$-partite. In this case, since $G$ is not $k$-partite, there must be an edge inside $T$, so that there exists $y,z$ as above. Hence $e_a(G)< e_a(G')$, and so clearly $e_a(G) < e_a(H)$ in this case too, as required.
\end{proof}

We now prove a stability version of Lemma~\ref{lem:turan}. The proof idea builds on that of Lemma~\ref{lem:turan}. The proof will also make use of the following proposition, which can be proved by a simple but tedious calculation, which we omit here\COMMENT{See Appendix}.

\begin{proposition}\label{omitted proposition}
Let $k, n\in \mathbb{N}$ with $n\geq k\geq 2$ and let $s >0$. Suppose $G$ is a $k$-partite graph on $n$ vertices in which some vertex class $A$ satisfies $|A-n/k|\geq s$. Then
$$e(G)\leq t_k(n)-s\left( \frac{s}{2}-k\right).$$
\end{proposition}

\begin{lemma}[Stability]\label{lem:sta}
Let $a\in \mathbb{R}$ with $3/2< a\leq 2$ and let $k\in \mathbb{N}$. For any $\beta > 0$ there exists $\g>0$ such that the following holds for all sufficiently large $n$. If a digraph $G$ on $n$ vertices is $T_{k+1}$-free, and $e_a(G) \ge {\rm ex}_a(n, T_{k+1}) - \g n^2$, then $G= DT_{k}(n) \pm \beta n^2$.
\end{lemma}

\begin{proof}
Choose $\g$ and $n_0$ such that $1/n_0 \ll \g \ll \beta$, and consider any $n\geq n_0$. We follow the proof of Lemma~\ref{lem:turan}, in which we fix a vertex $x\in V$
with $d^+(x)= \Delta^0(G)$, and let $S:= N^+(x)$ and $T:= V\setminus S$, and proceed by induction on $k$. Again, the base case $k=1$ is trivial as the only $T_2$-free digraph is the empty graph.
%Since $\te(G) \ge \ext(n, T_k) - \g n^2$, by averaging, we must have $|S|\ge (\frac{k-2}{k-1} - o(1))n$.
Let $m_1$ be the number of edges of $G[T]$, and let $m_2$ be the number of non-edges between $T$ and $S$ (in $G$).

Let $G'$ be the digraph obtained from $G$ by replacing each edge inside $T$ with two edges between $T$ and $S$ as in the proof of Lemma~\ref{lem:turan}. Then
\begin{equation}\label{stackrel equation 1}
e_a(G) \le e_a(G') - (2a- 3) m_1.
\end{equation}
Since there are $m_2 - 2m_1$ non-edges between $T$ and $S$ in $G'$, and adding any one of them would increase $e_a(G')$ by at least $a-1$ (since $a\leq 2$), we have that
\begin{equation}\label{stackrel equation 2}
e_a(G') \le |T| |S| a + e_a(G[S]) - (m_2 - 2 m_1)(a-1).
\end{equation}
Let
\begin{equation}\label{stackrel equation 3}
m_3 := {\rm ex}_a(|S|, T_{k}) - e_a(G[S]).
\end{equation}
Then $m_3\geq 0$ because $G[S]$ is $T_{k}$-free.

Let $H$ be the $k$-partite digraph obtained from $DT_{k-1}(|S|)$ (on $S$) by adding the vertex set $T$ together with all the edges (in both directions) between $S$ and $T$. Then
\begin{equation}\label{stackrel equation 4}
e_a(H) = |T| |S| a + {\rm ex}_a(|S|, T_{k}).
\end{equation}
Altogether, this gives that
\begin{align*}
e_a(G) &\stackrel{(\ref{stackrel equation 1})}{\leq} e_a(G')-(2a-3)m_1 \stackrel{(\ref{stackrel equation 2})}{\leq} |T| |S| a + e_a(G[S]) - (m_2 - 2 m_1)(a-1) - (2a-3)m_1\\
&\stackrel{(\ref{stackrel equation 3})}{=} |T| |S| a + {\rm ex}_a(|S|, T_{k}) - m_3 - (m_2 - 2 m_1)(a-1) - (2a-3)m_1\\
&\stackrel{(\ref{stackrel equation 4})}{=} e_a(H)- m_3 - (m_2 - 2 m_1)(a-1) - (2a-3)m_1.
\end{align*}
Let $s:=| |T| - \frac{n}{k}|$. By Proposition~\ref{omitted proposition} we have that if $s\geq 4k$ then ${\rm ex}_a(n, T_{k+1}) \ge a\cdot t_{k}(n) \geq e_a(H) + as^2 /4$. So if $s\geq 4k$ then
$$e_a(G) \leq {\rm ex}_a(n, T_{k+1}) - \frac{as^2}4 - m_3 - (m_2 - 2 m_1)(a-1) - (2a-3)m_1,$$
and if $s< 4k$ then
$$e_a(G) \leq {\rm ex}_a(n, T_{k+1}) - m_3 - (m_2 - 2 m_1)(a-1) - (2a-3)m_1.$$
In either case, since ${\rm ex}_a(n, T_{k+1}) - \g n^2 \le e_a(G)$ by assumption, we have that\COMMENT{The bound on $m_1$ follows since $\gamma n^2 \geq m_1(2a-3)$, and the bound on $m_2$ follows since $\gamma n^2 \geq (m_2 - 2m_1)(a-1)\geq (m_2-((2\gamma n^2)/(2a-3)))(a-1)$, where in the last inequality we use the bound on $m_1$ just found.} $m_1 \le \frac{\g}{2a-3} n^2$, $m_2\le (\frac{\g}{a-1} + \frac{2\g}{2a-3}) n^2$, $m_3 \le \g n^2$, and $s^2 \le 4 \g n^2/a$.

Recall that $e_a(G[S]) = {\rm ex}_a(|S|, T_{k}) - m_3$. Hence we have by induction hypothesis that, since $\g \ll \beta$ and $|S|= \Delta^0(G)$ is sufficiently large, $G[S] = DT_{k-1}(|S|) \pm (\b/2) |S|^2$. Note that we can obtain the digraph $DT_{k}(n)$ from $G$ by removing $m_1$ edges inside $T$, adding $m_2$ edges between $T$ and $S$, changing at most $(\b/2) n^2$ edges inside $S$, and changing the adjacency of at most $s$ vertices. Thus
\[
G= DT_{k}(n) \pm (m_1 + m_2 + (\b/2) n^2 + 2sn),
\]
and so we have that $G= DT_{k}(n) \pm \beta n^2,$ as required.
\end{proof}

We also need the Digraph Removal Lemma of Alon and Shapira \cite{AlSh}.
\begin{lemma}[Removal Lemma]
\label{lem:rem}
For any fixed digraph $H$ on $h$ vertices, and any $\g > 0$ there exists $\eps'>0$ such that the following holds for all sufficiently large $n$. If a digraph $G$ on $n$ vertices contains at most $\eps' n^h$ copies of $H$, then $G$ can be made $H$-free by deleting at most $\g n^2$ edges.
\end{lemma}

We are now ready to combine Theorem~\ref{thm:contain} with Lemma~\ref{lem:sta} to show that almost all $T_{k+1}$-free oriented graphs and almost all $T_{k+1}$-free digraphs are almost $k$-partite.

\begin{lemma}\label{step 0}
For every $k\in \mathbb{N}$ with $k\geq 2$ and any $\a > 0$ there exists $\eps>0$ such that the following holds for all sufficiently large $n$.
\begin{enumerate}[{\rm (i)}]
\item All but at most $ f(n, T_{k+1}) 2^{-\eps n^2}$ $T_{k+1}$-free oriented graphs on $n$ vertices can be made $k$-partite by changing at most $\a n^2$ edges.
\item All but at most $ f^*(n, T_{k+1}) 2^{-\eps n^2}$ $T_{k+1}$-free digraphs on $n$ vertices can be made $k$-partite by changing at most $\a n^2$ edges.
\end{enumerate}
\end{lemma}
\begin{proof}
We only prove {\rm (i)} here; the proof of {\rm (ii)} is almost identical. Let $a:=\log 3$. Choose $n_0\in \mathbb{N}$ and $\varepsilon, \gamma, \beta >0$ such that $1/n_0 \ll \eps \ll \g \ll \b \ll \a, 1/k$. Let $\varepsilon':= 2\varepsilon$ and $n\geq n_0$. By Theorem~\ref{thm:contain} (with $T_{k+1}, n$ and $\varepsilon'$ taking the roles of $H, N$ and $\varepsilon$ respectively) there is a collection $\C$ of digraphs on vertex set $[n]$ satisfying properties (a)--(c). In particular, by (a), every $T_{k+1}$-free oriented graph on vertex set $[n]$ is contained in some digraph $G\in \C$. Let $\C_1$ be the family of all those $G\in \C$ for which $e_{\log 3}(G)\ge {\rm ex}_{\log 3}(n, T_{k+1}) - \eps' n^2$. Then the number of (labelled) $T_{k+1}$-free oriented graphs not contained in some $G\in \C_1$ is at most
\[
|\C| \, 2^{{\rm ex}_{\log 3}(n, T_{k+1}) - \eps' n^2} \le 2^{- \eps n^2} f(n, T_{k+1}),
\]
because $|\C|\leq 2^{n^{2-\varepsilon'}}$, by (c), and $f(n, T_{k+1})\ge 2^{{\rm ex}_{\log 3}(n, T_{k+1})}$. Thus it suffices to show that every digraph $G\in \C_1$ satisfies $G= DT_{k}(n) \pm \a n^2$. By (b), each $G\in \C_1$ contains at most $\eps' n^{k+1}$ copies of $T_{k+1}$. Thus by Lemma~\ref{lem:rem} we obtain a $T_{k+1}$-free digraph $G'$ after deleting at most $\g n^2$ edges from $G$. Then $e_{\log 3}(G')\ge {\rm ex}_{\log 3}(n, T_{k+1}) - (\eps' + \g) n^2$. We next apply Lemma~\ref{lem:sta} to $G'$ and derive that $G'= DT_{k}(n) \pm \beta n^2$. As a result, the original digraph $G$ satisfies $G = DT_{k}(n) \pm (\beta + \g) n^2$, and hence $G = DT_{k}(n) \pm \a n^2$ as required.
\end{proof}

\section{Exact structure of typical $T_{k+1}$-free oriented graphs and digraphs} \label{Section: Exact structure of $T_k$-free oriented graphs and digraphs}

From Section~\ref{Section: Rough structure of $T_k$-free oriented graphs and digraphs} we know that a typical $T_{k+1}$-free oriented graph is almost $k$-partite (and similarly for digraphs). In this section we use this information to show inductively that we can omit the `almost' in this statement (see Lemma~\ref{inductive step} and the proof of Theorem~\ref{thm:main1} at the end of this section). 
The use of induction to obtain such exact results from approximate ones has been a useful tool in the past 
(already used in~\cite{EKR}),
but involves obstacles which are specific to the problem at hand.
Lemma~\ref{inductive step} relies on several simple observations about the typical structure of almost $k$-partite oriented graphs and digraphs (see Lemmas~\ref{step 1},~\ref{step 2} and~\ref{step 3}).

Recall that $t_k(n)$ denotes the maximum number of edges in a $k$-partite (undirected) graph on $n$ vertices, i.e. the number of edges in the $k$-partite Tur\'an graph on $n$ vertices. We say that a $k$-partition of vertices is \textit{balanced} if the sizes of any two partition classes differ by at most one. Given a $k$-partition $Q$ of $[n]$ with partition classes $V_1,\dots, V_k$, and a graph, oriented graph or digraph $G=(V,E)$ on vertex set $[n]$, and an edge $e=uv\in E$ with $u\in V_i$ and $v\in V_j$, we call $e$ a \emph{crossing edge} if $i\ne j$. In Lemma~\ref{bounds} below we give upper and lower bounds on $T(n,k)$ and $T^*(n,k)$, in terms of $t_k(n)$ (recall that $T(n,k)$ and $T^*(n,k)$ were defined in Section~\ref{Section: Notation and tools}). Lemma~\ref{bounds} is used in the proof of Theorem~\ref{thm:main1}.

\begin{lemma}\label{bounds}
Let $k\geq 2$. For sufficiently large $n$ we have the following:
\begin{enumerate}[{\rm (i)}]
\item $\frac{k^n3^{t_k(n)}}{2k!n^{k-1}}\leq  \frac{1}{2k!}\binom{n}{\left\lfloor \frac{n}{k} \right\rfloor, \dots, \left\lfloor \frac{n+k-1}{k} \right\rfloor}3^{t_k(n)}< T(n,k)< k^n3^{t_k(n)}.$
\item $\frac{k^n4^{t_k(n)}}{2k!n^{k-1}}< T^*(n,k)< k^n4^{t_k(n)}.$
\end{enumerate}
\end{lemma}
\begin{proof}
We only prove {\rm (i)} here; the proof of {\rm (ii)} is similar. For the upper bound note that $k^n$ counts the number of ordered $k$-partitions of $[n]$, and that for each such $k$-partition $Q$ the number of oriented graphs for which every edge is a crossing edge with respect to $Q$ is at most $3^{t_k(n)}$.\COMMENT{Indeed, the number of (unordered) pairs of vertices in different vertex classes of $Q$ is at most $t_k(n)$, and for each such pair there are $3$ possibilities (namely that it can be directed one way, directed the other way, or not present).}

For the lower bound we will count the number of (unordered) balanced $k$-partitions. Each such $k$-partition gives rise to $3^{t_k(n)}$ $k$-partite oriented graphs. Since the vertex classes of a balanced $k$-partition of $[n]$ have sizes $\left\lfloor \frac{n}{k} \right\rfloor, \dots, \left\lfloor \frac{n+k-1}{k} \right\rfloor$, the number of such $k$-partitions is
$$\frac{1}{k!}\binom{n}{\left\lfloor \frac{n}{k} \right\rfloor, \dots, \left\lfloor \frac{n+k-1}{k} \right\rfloor}.$$
We now show that for any given balanced $k$-partition $Q$, almost all $k$-partite oriented graphs for which $Q$ is a $k$-partition have no other possible $k$-partitions. Given a balanced $k$-partition $Q$ of $[n]$ with partition classes $A_1,\dots, A_k$, consider a random oriented graph where for each potential crossing edge we choose the edge to be either directed one way, directed the other way, or not present, each with probability $1/3$, independently. So each $k$-partite oriented graph for which $Q$ is a $k$-partition is equally likely to be generated. Given a set of vertices $A$ in a digraph $G$, we define their common out-neighbourhood $N^+(A):= \bigcap_{v\in A} N^+_G(v)$. By Theorem~\ref{Chernoff Bounds} we have that almost all graphs in the probability space satisfy the following:
\begin{enumerate}[($\alpha$)]
\item whenever $\ell \leq k$ and $i\in [k]$ and $v_1,\dots, v_{\ell}\in V(G)\setminus A_i$, we have that
$$|N^+(\{v_1,\dots, v_{\ell}\})\cap A_i|\geq (n/k)(1/3)^{\ell +1}.$$
\end{enumerate}
We now claim that if a $k$-partite oriented graph $G$ has $k$-partition $Q$ and satisfies $(\alpha)$ then $Q$ is the unique $k$-partition of $G$. Indeed, suppose that $Q'$ is a $k$-partition of $G$ with vertex classes $A_1',\dots, A_k'$. We will show that $Q'=Q$. Consider any $k$ vertices $v_1,\dots, v_k$ that are such that $G[\{v_1,\dots, v_k\}]$ is a transitive tournament. Such a set of $k$ vertices exists by $(\alpha)$. Clearly no two of these vertices can be in the same vertex class of $Q$ or $Q'$. Without loss of generality let us assume that $v_i\in A_i$ and $v_i\in A_i'$ for every $i\in [k]$. Define $N_i:= N^+(\{v_1,\dots, v_{k}\}\setminus \{v_i\})$. Since $N_i$ is the common out-neighbourhood of $\{v_1,\dots, v_{k}\}\setminus \{v_i\}$ it must be that $N_i$ is a subset of $A_i$ and a subset of $A_i'$. Note that $Q$ and $Q'$ agree on all vertices so far assigned to a partition class of $Q'$. Now consider any vertex $w$ not yet assigned to a partition class of $Q'$, and suppose $w\in A_j$ for some $j\in [k]$. For every $i\in [k]$ with $i \ne j$ we have by $(\alpha)$ that
$$|N^+(w)\cap N_i|=|N^+(\{w, v_1,\dots, v_k\} \setminus \{v_i\})\cap A_i|\geq (n/k)(1/3)^{k+1}\geq 1.$$
This together with the previous observation that $N_i\subseteq A_i'$ implies that $w\notin A_i'$. So $w\in A_j'$. Since $w$ was an arbitrary unassigned vertex we have that $A_i=A_i'$ for every $i\in [k]$, and so $Q=Q'$, which implies the claim. This completes the proof of the middle inequality in Lemma~\ref{bounds}.

To prove the first inequality note that if $a_1+\dots+a_k=n$ then $\binom{n}{a_1,\dots, a_k}$ is maximised by taking $a_j:=\left\lfloor \frac{n+j-1}{k} \right\rfloor$ for every $j\in [k]$. This implies that
$$k^n=\sum\limits_{a_1+\dots+a_k=n}\binom{n}{a_1,\dots, a_k}\leq n^{k-1}\binom{n}{\left\lfloor \frac{n}{k} \right\rfloor, \dots, \left\lfloor \frac{n+k-1}{k} \right\rfloor},$$
which in turn implies the first inequality in Lemma~\ref{bounds}, and hence completes the proof.
\end{proof}

For a given oriented graph or digraph $G$ on vertex set $[n]$ we call a $k$-partition $Q$ of $[n]$ \textit{optimal} if the number of non-crossing edges in $G$ with respect to $Q$ is at most the number of non-crossing edges in $G$ with respect to $Q'$ for every other $k$-partition $Q'$ of $[n]$.

Given $k\geq 2$ and $\eta>0$ we define $F(n, T_{k+1}, \eta)$ to be the set of all labelled $T_{k+1}$-free oriented graphs on $n$ vertices that have at most $\eta n^2$ non-crossing edges in an optimal $k$-partition. We define $F_Q(n, T_{k+1}, \eta)\subseteq F(n, T_{k+1}, \eta)$ to be the set of all such oriented graphs for which $Q$ is an optimal $k$-partition. Similarly, we define $F^*(n, T_{k+1}, \eta)$ to be the set of all labelled $T_{k+1}$-free digraphs on $n$ vertices that have at most $\eta n^2$ non-crossing edges in an optimal $k$-partition, and we define $F^*_Q(n, T_{k+1}, \eta)\subseteq F^*(n, T_{k+1}, \eta)$ to be the set of all such digraphs for which $Q$ is an optimal $k$-partition. Define 
$$f(n, T_{k+1}, \eta):=|F(n, T_{k+1}, \eta)| \hspace{0.6cm} \text{and} \hspace{0.6cm}  f_Q(n, T_{k+1}, \eta):=|F_Q(n, T_{k+1}, \eta)|,$$
and similarly
$$f^*(n, T_{k+1}, \eta):=|F^*(n, T_{k+1}, \eta)| \hspace{0.6cm} \text{and} \hspace{0.6cm}  f^*_Q(n, T_{k+1}, \eta):=|F^*_Q(n, T_{k+1}, \eta)|.$$
Then Lemma~\ref{step 0} implies that for every $\eta>0$ there exists $\varepsilon'>0$ such that
\begin{equation}\label{modified rough structure equation}
f(n, T_{k+1})\leq f(n, T_{k+1}, \eta)(1+2^{-\varepsilon' n^2}) \hspace{0.6cm}\text{and}\hspace{0.6cm} f^*(n, T_{k+1})\leq f^*(n, T_{k+1}, \eta)(1+2^{-\varepsilon' n^2})
\end{equation}
for all sufficiently large $n$. (So $\varepsilon'=2\varepsilon$, where $\varepsilon$ is as given by Lemma~\ref{step 0}.)

Given an oriented graph or digraph $G$ on vertex set $V$ and disjoint subsets $U, U'\subseteq V$ we let $\overrightarrow{e_G}(U, U')$ denote the number of edges in $E(G)$ directed from vertices in $U$ to vertices in $U'$. For convenience we will sometimes write $\overrightarrow{e}(U, U')$ for $\overrightarrow{e_G}(U, U')$ if this creates no ambiguity. Given $k\in \mathbb{N}$, $\eta, \mu>0$, and a $k$-partition $Q$ of $[n]$ with vertex classes $A_1,\dots, A_k$, we define $F_Q(n, \eta, \mu)$ (respectively $F^*_Q(n, \eta, \mu)$) to be the set of all labelled oriented (respectively directed) graphs on $n$ vertices for which $Q$ is an optimal $k$-partition and that satisfy the following:
\begin{enumerate}[(F1)]
\item the number of non-crossing edges with respect to $Q$ is at most $\eta n^2$,
\item if $U_i\subseteq A_i$ and $U_j\subseteq A_j$ with $|U_i|, |U_j|\geq \mu n$ for distinct $i,j\in [k]$, then $\overrightarrow{e}(U_i, U_j), \overrightarrow{e}(U_j, U_i)\geq |U_i||U_j|/6$,
\item $||A_i|-n/k|\leq \mu n$ for every $i\in [k]$.
\end{enumerate}
Note that property (F2) is similar to the property that the bipartite graph on vertex classes $A_i, A_j$ whose edges are directed from $A_i$ to $A_j$ is $\mu$-regular of density at least $1/6$ (and similarly for edges directed from $A_j$ to $A_i$) and that the `reduced graph' $R$ that has vertex set $\{A_1,\dots, A_k \}$ and edges between pairs that are $\mu$-regular of density at least $1/6$ is a complete digraph.

Define $F_Q(n, T_{k+1}, \eta, \mu)$ to be the set of all oriented graphs in $F_Q(n, \eta, \mu)$ that are $T_{k+1}$-free. Similarly define $F_Q^*(n, T_{k+1}, \eta, \mu)$ to be the set of all digraphs in $F_Q^*(n, \eta, \mu)$ that are $T_{k+1}$-free. Note that $F_Q(n, T_{k+1}, \eta, \mu) \subseteq F^*_Q(n, T_{k+1}, \eta, \mu)$. Define $f_Q(n, T_{k+1}, \eta, \mu):=|F_Q(n, T_{k+1}, \eta, \mu)|$ and $f^*_Q(n, T_{k+1}, \eta, \mu):=|F^*_Q(n, T_{k+1}, \eta, \mu)|$.

The next lemma shows that $f_Q(n, T_{k+1}, \eta)$ and $f_Q(n, T_{k+1}, \eta, \mu)$ are asymptotically equal for any $k$-partition $Q$ and suitable parameter values, (and similarly for $f^*_Q(n, T_{k+1}, \eta)$ and $f^*_Q(n, T_{k+1}, \eta, \mu)$).

\begin{lemma}\label{step 1}
Let $k\geq 2$ and let $0<\eta, \mu<1$ be such that $\mu^2\geq 24H(\eta)$. There exists an integer $n_0=n_0(\mu, k)$ such that the following hold for all $n\geq n_0$ and for every $k$-partition $Q$ of $[n]$:
\begin{enumerate}[{\rm (i)}]
\item $f_Q(n, T_{k+1}, \eta)-f_Q(n, T_{k+1}, \eta, \mu)\leq 3^{t_k(n)-\frac{\mu^2 n^2}{100}}.$
\item $f^*_Q(n, T_{k+1}, \eta)-f^*_Q(n, T_{k+1}, \eta, \mu)\leq 4^{t_k(n)-\frac{\mu^2 n^2}{100}}.$
\end{enumerate}
\end{lemma}
\begin{proof}
We only prove {\rm (i)} here; the proof of {\rm (ii)} is similar. We choose $n_0$ such that $1/n_0 \ll \mu, 1/k$. We wish to count the number of $G\in F_Q(n, T_{k+1}, \eta)\setminus F_Q(n, T_{k+1}, \eta, \mu)$. Let $Q$ have vertex classes $A_1,\dots, A_k$. The number of ways that at most $\eta n^2$ non-crossing edges can be placed is at most
$$\binom{n^2}{\leq \eta n^2}\stackrel{(\ref{entropy bound})}{\leq} 2^{H(\eta) n^2}.$$

If $||A_i|-n/k|>\mu n$ for some $i\in [k]$ then by Proposition~\ref{omitted proposition} the number of possible crossing edges is at most
$$t_k(n)-\mu n\left( \frac{\mu n}{2}-k\right)\leq t_k(n)-\frac{\mu^2 n^2}{3}.$$
We can conclude that the number of $G\in F_Q(n, T_{k+1}, \eta)\setminus F_Q(n, T_{k+1}, \eta, \mu)$ that fail to satisfy (F3) is at most
$$2^{H(\eta) n^2}3^{t_k(n)-\mu ^2 n^2 /3 }.$$

Every $G\in F_Q(n, T_{k+1}, \eta)\setminus F_Q(n, T_{k+1}, \eta, \mu)$ that satisfies property (F3) must fail to satisfy property (F2). For a given choice of at most $\eta n^2$ non-crossing edges, consider the random oriented graph $H$ where for each possible crossing edge with respect to $Q$ we choose the edge to be either directed in one direction, directed in the other direction, or not present, each with probability $1/3$, independently. Note that the total number of ways to choose the crossing edges is at most $3^{t_k(n)}$, and each possible configuration of crossing edges is equally likely. So an upper bound on the number of $G\in F_Q(n, T_{k+1}, \eta)\setminus F_Q(n, T_{k+1}, \eta, \mu)$ that fail to satisfy property (F2) is
$$2^{H(\eta) n^2}3^{t_k(n)}\mathbb{P}(H \hspace{0.09cm} \text{\textnormal{on}}\hspace{0.09cm} A_1,\dots, A_k \hspace{0.09cm} \text{\textnormal{fails to satisfy (F2)}}).$$
Note that the number of choices for $U_i\subseteq A_i$ and $U_j\subseteq A_j$ as in property (F2) is at most $(2^n)^2$ and that $\mathbb{E}(\overrightarrow{e_H}(U_i, U_j))= |U_i||U_j|/3 \geq \mu^2 n^2 /3$. Hence by Theorem~\ref{Chernoff Bounds} we get that
$$\mathbb{P}(H \hspace{0.09cm} \text{\textnormal{on}}\hspace{0.09cm} A_1,\dots, A_k \hspace{0.09cm} \text{\textnormal{fails to satisfy (F2)}})\leq (2^n)^2\exp\left(- \frac{\mathbb{E}(\overrightarrow{e_H}(U_i, U_j))}{8} \right)\leq 2^{2n}\exp \left(-\frac{\mu^2 n^2}{24}\right).$$
So summing these upper bounds gives us that
\begin{align*}
f_Q(n, T_{k+1}, \eta)-f_Q(n, T_{k+1}, \eta, \mu)&\leq 2^{H(\eta) n^2}3^{t_k(n)}\left( 3^{-\mu^2 n^2 /3}+2^{2n}e^{-\frac{\mu^2 n^2}{24}}\right)\\
&\leq 3^{t_k(n)} 3^{-\frac{\mu^2 n^2}{24}(\log_3 e-\log_3 2)}2^{2n+1}  \leq 3^{t_k(n)-\frac{\mu^2 n^2}{100}},
\end{align*}
where we use that $\mu^2\geq 24H(\eta)$ and that $1/n_0 \ll \mu$.
\end{proof}

The following proposition allows us to find many disjoint copies of $T_k$ in any graph in $F^*_Q(n, \eta, \mu)$. It will be useful in proving Lemmas~\ref{step 2} and \ref{inductive step}. We omit the proof, since it amounts to embedding a small oriented subgraph into the $\mu$-regular blow-up of a complete digraph which can be done greedily\COMMENT{
First note that, since $G\in F^*_Q(n,\eta, \mu)$, if $U_1\subseteq A_i$ and $U_2\subseteq A_j$, where $i,j\in [k]$ are distinct, and $|U_1|, |U_2|\geq \mu n$ then $\overrightarrow{e}(U_1, U_2)\geq |U_1||U_2|/6$. So there must be at least $|U_1|/12$ vertices in $U_1$ that have at least $|U_2|/12$ out-neighbours in $U_2$. Indeed, if not then $\overrightarrow{e}(U_1, U_2)$ is less than
$$(|U_1|/12 \times |U_2|)+(|U_1|\times |U_2|/12)= |U_1||U_2|/6,$$
which is a contradiction.
By this observation, since $|B_{\sigma(1)}|, |B_{\sigma(2)}|\geq 12^{k-2}\mu n\geq \mu n,$ there is a set $B^{(2)}_{\sigma(1)}$, with size at least $|B_{\sigma(1)}|/12$, of vertices in $B_{\sigma(1)}$ that have at least $|B_{\sigma(2)}|/12$ out-neighbours in $B_{\sigma(2)}$. Note that if $k=2$ then we are now done, since any vertex $x$ in $B^{(2)}_{\sigma(1)}$ and any out-neighbour of $x$ in $B_{\sigma(2)}$ form a $T_2$ with the edge directed as required. Otherwise, since $|B^{(2)}_{\sigma(1)}|, |B_{\sigma(3)}|\geq 12^{k-3}\mu n\geq \mu n$ there is a set $B^{(3)}_{\sigma(1)}$, with size at least $|B_{\sigma(1)}|/(12^2)$, of vertices in $B^{(2)}_{\sigma(1)}$ that have at least $|B_{\sigma(3)}|/12$ out-neighbours in $B_{\sigma(3)}$. Since $|B_{\sigma(1)}|\geq 12^{k-2}\mu n$ we may continue this process to find at least $\mu n/12$ vertices in $B_{\sigma(1)}$ with at least $|B_{\sigma(m)}|/12$ out-neighbours in $B_{\sigma(m)}$, for every $1< m\leq k$. Fix one such vertex, $v_{\sigma(1)}$, and denote its out-neighbourhood in $B_{\sigma(m)}$ by $B^{(2)}_{\sigma(m)}$. Since $|B^{(2)}_{\sigma(2)}|\geq 12^{k-3}\mu n$ we may similarly find a vertex $v_{\sigma(2)}$ in $B^{(2)}_{\sigma(2)}$ that has at least $|B^{(2)}_{\sigma(m)}|/12$ out-neighbours in $B^{(2)}_{\sigma(m)}$, for every $2< m\leq k$. We may now continue this process to find vertices $v_1,\dots, v_k$ that form a copy of $T_k$ with edges directed as required.} (see e.g.~\cite[Lemma 7.5.2]{Diestel} for the `undirected' argument).

\begin{proposition}\label{finding T_k}
Let $n, k\in \mathbb{N}$, let $\eta, \mu >0$, let $Q$ be a $k$-partition of $[n]$ with vertex classes $A_1,\dots, A_k$, and suppose $G\in F^*_Q(n,\eta, \mu)$. For every $i\in [k]$ let $B_i\subseteq A_i$ with $|B_i|\geq 12^{k-2}\mu n$. Let $\sigma$ be a permutation of $[k]$. Then $G$ contains a copy of $T_k$ on vertices $v_1,\dots, v_k$ where for all distinct $i,j\in [k]$ we have that $v_i\in B_i$ and that there is an edge from $v_i$ to $v_j$ if and only if $\sigma(i)<\sigma(j)$.
\end{proposition}

We now show that in an optimal partition each vertex is contained in only a small number of non-crossing edges.

\begin{lemma}\label{step 2}
Let $n, k\geq 2$, let $\eta, \mu >0$, let $Q$ be a $k$-partition of $[n]$ with vertex classes $A_1,\dots, A_k$, and suppose $G\in F^*_Q(n,T_{k+1},\eta, \mu)$. Then for every $i\in [k]$ and every $x\in A_i$ we have that
$$|N^+_{A_i}(x)|+|N^-_{A_i}(x)|\leq 12^{k-2}2\mu n.$$
\end{lemma}
\begin{proof}
Suppose not, so that there exists $x\in A_i$, for some $i\in [k]$, such that $|N^+_{A_i}(x)|+|N^-_{A_i}(x)|> 12^{k-2}2\mu n$. Since $Q$ is an optimal $k$-partition of $G$, it must be that
$$|N^+_{A_j}(x)|+|N^-_{A_j}(x)|\geq |N^+_{A_i}(x)|+|N^-_{A_i}(x)|> 12^{k-2}2\mu n$$
for every\COMMENT{If this inequality fails for some $j\in [k]\setminus \{i\}$ we can construct a partition with less non-crossing edges than our optimal partition by moving $x$ from $A_i$ to $A_j$, which contradicts the definition of optimality.} $j\in [k]$.

For every $j\in [k]$ define $B_j$ to be $N^+_{A_j}(x)$ if $|N^+_{A_j}(x)|\geq|N^-_{A_j}(x)|$, and $N^-_{A_j}(x)$ otherwise. So $|B_j|\geq 12^{k-2}\mu n$. Let $J^+$ be the set of all $j\in [k]$ such that $B_j=N^+_{A_j}(x)$, and let $J^-:= [k]\setminus J^+$. Fix a permutation $\sigma$ of $[k]$ with the property that $\sigma(i)<\sigma(j)$ whenever $i\in J^-$ and $j\in J^+$. Now Proposition~\ref{finding T_k} implies that $G$ contains a copy of $T_k$ on vertices $v_1,\dots, v_k$ where for all distinct $i,j\in [k]$ we have that $v_i\in B_i$ and that the edge between $v_i$ and $v_j$ is directed towards $v_j$ if and only if $\sigma(i)<\sigma(j)$. By the definition of $\sigma$, $x$ together with this copy of $T_k$ forms a copy of $T_{k+1}$. This is a contradiction, since $G\in F^*_Q(n,T_{k+1},\eta, \mu)$, and so this completes the proof.
\end{proof}

The following result shows that an optimal partition of a graph does not change too much upon the removal of just two vertices from the graph.

\begin{lemma}\label{step 3}
Let $k\geq 2$, and let $0<\mu<1/(3k^2)^{12}$ and $0< \eta < \mu^2/3$. There exists an integer $n_0= n_0(\mu, k)$ such that the following holds for all $n\geq n_0$. Let $Q$ be a partition of $[n]$ with vertex classes $A_1,\dots, A_k$ and let $x,y$ be distinct elements of $A_1$. Then there is a set $\mathcal{P}$ of $k$-partitions of $[n]\setminus \{x,y\}$, with $|\mathcal{P}|\leq e^{\mu^{2/3} n}$, such that, for every $G\in F^*_Q(n,T_{k+1},\eta, \mu)$, every optimal $k$-partition of $G-\{x,y\}$ is an element of $\mathcal{P}$.
\end{lemma}
\begin{proof}
First note that, for any $G\in F^*_Q(n,T_{k+1},\eta, \mu)$, we have by definition that the number of non-crossing edges in $G$ with respect to $Q$ is at most $\eta n^2$. So certainly the number of non-crossing edges in $G-\{x,y\}$ with respect to the partition $A_1\setminus \{x,y\}, A_2\dots, A_k$ is at most $\eta n^2$.

Consider an arbitrary $k$-partition $B_1,\dots, B_k$ of $[n]\setminus \{x,y\}$. We claim that if there exists $i\in [k]$ and distinct $j, j'\in [k]$ such that $|A_j\cap B_i|, |A_{j'}\cap B_i|\geq \mu n$, then for any $G\in F^*_Q(n,T_{k+1},\eta, \mu)$ the number of non-crossing edges in $G-\{x,y\}$ with respect to the partition $B_1,\dots, B_k$ is larger than $\eta n^2$ (and hence $B_1,\dots, B_k$ cannot be an optimal $k$-partition of $G-\{x,y\}$). Indeed, if we find such $i, j, j'$ then by (F2) we have that the number of non-crossing edges in $G-\{x,y\}$ with respect to the partition $B_1,\dots, B_k$ is at least
$$e_G(B_i)\geq e_G(A_j\cap B_i, A_{j'}\cap B_i)\geq 2\cdot \frac{1}{6}(\mu n)^2>\eta n^2,$$
which proves the claim.

We let $\mathcal{P}$ be the set of all $k$-partitions of $[n]\setminus \{x,y\}$ for which no such $i, j, j'$ exist. So by the above claim we have that for every $G\in F^*_Q(n,T_{k+1},\eta, \mu)$, every optimal $k$-partition of $G-\{x,y\}$ is an element of $\mathcal{P}$. So it remains to show that $|\mathcal{P}|\leq e^{\mu^{2/3} n}$. Consider an element of $\mathcal{P}$ with partition classes $B_1,\dots, B_k$. For every $i\in [k]$, let $S_i:=\{j:|A_j\cap B_i|\geq \mu n\}$. Note that for every $i\in [k]$ we have that $|S_i|\leq 1$, by definition of $\mathcal{P}$. Note also that $|A_j|\geq n/k-\mu n > k\mu n$ for every $j\in [k]$, and thus for every $i\in [k]$ we have that $|S_i|=1$. Let $A_1':= A_1\setminus \{x,y\}$ and let $A_i':= A_i$ for every $i\in \{2, \dots, k\}$. So every element of $\mathcal{P}$ can be obtained by starting with the $k$-partition $A_1',\dots, A_k'$, applying a permutation of $[k]$ to the partition class labels, and then for every ordered pair of partition classes moving at most $\mu n$ elements from the first partition class to the second. Hence, since $|A_j|\leq n/k +\mu n\leq 2n/k$, we have that\COMMENT{
$k! (\mu n)^{k^2} \left( \frac{1}{\mu^2} \right)^{\mu k^2 n} \leq n^{k^2}e^{2(\log (1/ \mu))\mu k^2 n}\leq n^{k^2}e^{2((1/ \mu)^{1/4})\mu k^2 n} \leq e^{\mu^{2/3} n},$ where the last inequality holds since $\mu < 1/(3k^2)^{12}$.}
\begin{align*}
|\mathcal{P}|&\leq k!\left( \binom{2n/k}{\leq \mu n}^{k-1} \right)^k\leq k! \left( \mu n \left( \frac{2en/k}{\mu n} \right)^{\mu n} \right)^{k(k-1)}\leq k! (\mu n)^{k^2} \left( \frac{1}{\mu^2} \right)^{\mu k^2 n} \leq e^{\mu^{2/3} n},
\end{align*}
as required.
\end{proof}

Define $F_Q'(n,T_{k+1}, \eta)$ to be the set of all oriented graphs in $F_Q(n,T_{k+1}, \eta)$ that have at least one non-crossing edge with respect to $Q$. Define $f_Q'(n,T_{k+1}, \eta) := |F_Q'(n,T_{k+1}, \eta)|$. Similarly define $F^{*'}_Q(n,T_{k+1}, \eta)$ to be the set of all digraphs in $F^*_Q(n,T_{k+1}, \eta)$ that have at least one non-crossing edge with respect to $Q$, and define $f^{*'}_Q(n,T_{k+1}, \eta) := |F^{*'}_Q(n,T_{k+1}, \eta)|$. In the following result we use Lemmas~\ref{step 0},~\ref{step 1}, ~\ref{step 2} and \ref{step 3} to give upper bounds on $f_Q'(n,T_{k+1}, \eta)$ and $f^{*'}_Q(n,T_{k+1}, \eta)$ for any $k$-partition $Q$ and suitable parameter values.

\begin{lemma}\label{inductive step}
For all $k\geq 2$ there exist $\eta>0$ and $C\in \mathbb{N}$ such that for all $n\in \mathbb{N}$ and all $k$-partitions $Q$ of $[n]$ the following hold.
\begin{enumerate}[{\rm (i)}]
\item $f_Q'(n,T_{k+1}, \eta)\leq 3^{t_k(n)} C2^{-\eta n}.$
\item $f^{*'}_Q(n,T_{k+1}, \eta)\leq 4^{t_k(n)} C2^{-\eta n}.$
\end{enumerate}
\end{lemma}
\begin{proof}
We only prove {\rm (i)} here; the proof of {\rm (ii)} is similar. Choose $C, n_0\in \mathbb{N}$ and $\varepsilon, \eta, \mu >0$ such that
$$1/C \ll 1/n_0 \ll \varepsilon \ll \eta \ll \mu \ll 1/k.$$
Define $F_Q(n,T_{k+1})$ to be the set of all $T_{k+1}$-free oriented graphs on $n$ vertices for which $Q$ is an optimal $k$-partition, and define $f_Q(n,T_{k+1}) = |F_Q(n,T_{k+1})|$.

The proof proceeds by induction on $n$. In fact, in addition to {\rm (i)} we will inductively show that
\begin{equation}\label{inductive aim 1}
f_Q(n,T_{k+1})\leq 3^{t_k(n)} (1+ C2^{-\eta n}).
\end{equation}
The result holds trivially for $n< n_0$ since $1/C \ll 1/n_0$. So let $n\geq n_0$ and let us assume that for every $k$-partition $Q'$ of $[n-2]$ we have that
\begin{equation}\label{inductive assumption}
f_{Q'}(n-2,T_{k+1})\leq 3^{t_k(n-2)}(1+ C2^{-\eta(n-2)}).
\end{equation}

Let $Q$ have partition classes $A_1,\dots, A_k$. Define $F_Q'(n,T_{k+1}, \eta, \mu)$ to be the set of all graphs in $F_Q(n,T_{k+1}, \eta, \mu)$ that have at least one non-crossing edge with respect to $Q$. Define $f_Q'(n,T_{k+1}, \eta, \mu) := |F_Q'(n,T_{k+1}, \eta, \mu)|$. We will first find an upper bound for $f_Q'(n,T_{k+1}, \eta, \mu)$. We will find this bound in four steps. Note that (F3) implies that $f_Q'(n,T_{k+1}, \eta, \mu)=0$ unless $||A_i|-n/k|\leq \mu n$, so we may assume that this inequality holds.

\vspace{0.3cm}
\noindent {\bf Step 1:} Let $B_1$ be the number of ways to choose a single non-crossing edge $xy$ with respect to $Q$. Then $B_1\leq n^2$. Let $A_i$ be the partition class of $Q$ containing $x$ and $y$.

\vspace{0.3cm}
\noindent {\bf Step 2:} Let $B_2$ be the number of ways to choose the edges that do not have an endpoint in $\{x, y\}$. By Lemma~\ref{step 3} there is a set $\mathcal{P}$ of $k$-partitions of $[n]\setminus \{x,y\}$, with $|\mathcal{P}|\leq e^{\mu^{2/3} n}$, such that, for every $G\in F_Q(n,T_{k+1},\eta, \mu)$, every optimal $k$-partition of $G-\{x,y\}$ is an element of $\mathcal{P}$. So we have by our inductive hypothesis that
$$B_2\leq \sum\limits_{Q'\in \mathcal{P}} f_{Q'}(n-2, T_{k+1})\stackrel{(\ref{inductive assumption})}{\leq} e^{\mu^{2/3} n} 3^{t_k(n-2)}(1+ C2^{-\eta(n-2)})\leq 3^{t_k(n-2)}Ce^{\mu^{1/2}n}.$$

\noindent {\bf Step 3:} Let $B_3$ be the number of possible ways to construct the edges between $x, y$ and the vertices outside $A_i$. Let $U$ be the set of edges chosen in Step $2$. One can view $U$ as a subset of the edge set of a graph $G$ in $F_Q'(n,T_{k+1},\eta, \mu)$. Let $U'$ be the subset of $U$ consisting of all those edges in $U$ that do not have an endpoint in $A_i$. So $U'$ can be viewed as the edge set of a subgraph $G'$ of $G$ with $G'\in F_{\tilde{Q}}(n-|A_i|,5\eta, 3\mu)$, where the set of partition classes of $\tilde{Q}$ is $\{A_1,\dots, A_k\}\setminus \{A_i\}$. By repeatedly applying Proposition~\ref{finding T_k} to $G'$ we can find at least $n/k-\mu n -12^{k-3}3\mu n$ vertex-disjoint copies of $T_{k-1}$ in $G'$, each with precisely one vertex in each of the $A_j$ for $j\ne i$. Consider the $2(k-1)$ potential edges between $x, y$ and the vertices of any such $T_{k-1}$. If we wish for our graph to remain $T_{k+1}$-free then not all of the possible $3^{2(k-1)}$ sets of such edges are allowed. So since the number of vertices outside $A_i$ not contained in one of these $T_{k-1}$ is at most $(k-1)(2\mu n +12^{k-3}3\mu n)\leq \mu^{1/2}n \log_3 e/2$, we have that
$$B_3\leq \left( 3^{2(k-1)}-1\right)^{n/k}3^{2 \left( \mu^{1/2}n \log_3 e/2 \right) }< \left( 3^{2(k-1)} \left( 1-3^{-2k}\right) \right)^{n/k} e^{ \mu^{1/2}n }\leq 3^{2\frac{k-1}{k}n} e^{-\frac{n}{9^k k}}e^{ \mu^{1/2}n}.$$

\noindent {\bf Step 4:} Let $B_4$ be the number of possible ways to construct the edges between $x, y$ and the other vertices in $A_i$. Note that by Lemma~\ref{step 2}, $x$ and $y$ each have at most $12^{k-2}2\mu n$ neighbours inside $A_i$, and for each of these the edge between them may be oriented in either direction. So since $|A_i|\leq n$, we have that
\begin{align*}
B_4&\leq \binom{n}{\leq 12^{k-2}2\mu n}^2 \left(2^{12^{k-2}2\mu n}\right)^2\leq \binom{n}{12^{k-2}2\mu n}^2 \left( 2^{12^{k-2}2\mu n} \right)^2 \left(2^{12^{k-2}2\mu n}\right)^2\\
&\leq \left( \frac{2e}{12^{k-2}\mu}\right)^{12^{k-2}4\mu n}\leq e^{\mu^{1/2}n}.
\end{align*}

In Steps $1$--$4$ we have considered all possible edges, and so $f_Q'(n,T_{k+1}, \eta, \mu)\leq B_1\cdot B_2\cdot B_3\cdot B_4$. Together with the fact that $t_k(n)\geq t_k(n-2)+2((k-1)/k)(n-2)$ this implies that
\begin{align}\label{complicated equation}
f_Q'(n,T_{k+1}, \eta, \mu)&\leq n^2 3^{t_k(n-2)}C3^{2((k-1)/k)n}e^{3\mu^{1/2} n} e^{-n/(k9^k)}\\
&\leq 3^{t_k(n)}Ce^{-n/(2k9^k)} \leq 3^{t_k(n)} C2^{-3\eta n}\nonumber.
\end{align}
Now, note that since $F_Q(n,T_{k+1}, \eta, \mu), F_Q'(n,T_{k+1}, \eta) \subseteq F_Q(n,T_{k+1}, \eta)$ we have that
\begin{align*}
f_Q'(n,T_{k+1}, \eta)&= |F_Q'(n,T_{k+1}, \eta)\cap F_Q(n,T_{k+1}, \eta, \mu)| + |F_Q'(n,T_{k+1}, \eta) \setminus F_Q(n,T_{k+1}, \eta, \mu)|\\
&= |F_Q'(n,T_{k+1}, \eta, \mu)| + |F_Q'(n,T_{k+1}, \eta) \setminus F_Q(n,T_{k+1}, \eta, \mu) |\\
&\leq |F_Q'(n,T_{k+1}, \eta, \mu)| + |F_Q(n,T_{k+1}, \eta) \setminus F_Q(n,T_{k+1}, \eta, \mu) |\\
&= f_Q'(n,T_{k+1}, \eta, \mu) + \left( f_Q(n,T_{k+1}, \eta) - f_Q(n,T_{k+1}, \eta, \mu) \right).
\end{align*}
This together with Lemma~\ref{step 1}{\rm (i)} gives us that
$$f_Q'(n,T_{k+1}, \eta)\leq f_Q'(n,T_{k+1}, \eta, \mu)+ 3^{t_k(n)-\frac{\mu ^2 n^2}{100}} \stackrel{(\ref{complicated equation})}{\leq} 3^{t_k(n)}C2^{-\eta n},$$
which proves {\rm (i)}. So it remains to prove (\ref{inductive aim 1}).

Note that the number of graphs in $F_Q(n,T_{k+1}, \eta, \mu)$ for which every edge is a crossing edge with respect to $Q$ is at most $3^{t_k(n)}$. Since $f_Q(n,T_{k+1}, \eta, \mu)-f_Q'(n,T_{k+1}, \eta, \mu)$ is precisely the number of such graphs, we have that
$$f_Q(n,T_{k+1}, \eta, \mu)-f_Q'(n,T_{k+1}, \eta, \mu) \leq 3^{t_k(n)}.$$
This together with (\ref{complicated equation}) implies that
$$f_Q(n,T_{k+1}, \eta, \mu)\leq 3^{t_k(n)}\left( 1+C2^{-3\eta n} \right).$$
Together with Lemma~\ref{step 1}{\rm (i)} this implies that
\begin{equation}\label{flower equation}
f_Q(n,T_{k+1}, \eta)\leq f_Q(n,T_{k+1}, \eta, \mu)+ 3^{t_k(n)-\frac{\mu ^2 n^2}{100}} \leq 3^{t_k(n)}\left( 1+C2^{-2\eta n} \right).
\end{equation}
On the other hand, Lemma~\ref{step 0}{\rm (i)} implies that
\begin{eqnarray}
f(n, T_{k+1})-f(n, T_{k+1}, \eta) \hspace{-0.3cm} & \leq & \hspace{-0.3cm} f(n, T_{k+1})2^{-\varepsilon n^2} \stackrel{(\ref{modified rough structure equation})}{\leq} 2 f(n, T_{k+1}, \eta)2^{-\varepsilon n^2}\nonumber\\
& \stackrel{(\ref{flower equation})}{\leq} & \hspace{-0.3cm} 2k^n 3^{t_k(n)}\left( 1+C2^{-2\eta n} \right)2^{-\varepsilon n^2}\leq 3^{t_k(n)} C2^{-2\eta n}.\nonumber
\end{eqnarray}
Now this together with (\ref{flower equation}) implies that
$$f_Q(n, T_{k+1})\leq f_Q(n, T_{k+1}, \eta) + (f(n, T_{k+1})-f(n, T_{k+1}, \eta))\leq 3^{t_k(n)} \left( 1+ C2^{-\eta n} \right).$$
This completes the proof of~(\ref{inductive aim 1}).
\end{proof}

We can now finally prove Theorem~\ref{thm:main1} using Lemmas~\ref{bounds} and~\ref{inductive step} together with the bounds in~(\ref{modified rough structure equation}).

\removelastskip\penalty55\medskip\noindent{\bf Proof of Theorem~\ref{thm:main1}.}
We only prove {\rm (i)} here; the proof of {\rm (ii)} is similar. Let $\eta, C$ be given by Lemma~\ref{inductive step} and choose $n_0$ and $\varepsilon$ such that $1/n_0 \ll \varepsilon \ll \eta, 1/k$. Consider any $n\geq n_0$ and let $\mathcal{Q}$ be the set of all $k$-partitions of $[n]$. Note that
$$f'(n, T_{k+1}, \eta)\leq \sum\limits_{Q\in \mathcal{Q}} f_Q'(n, T_{k+1}, \eta).$$
So by Lemma~\ref{inductive step}{\rm (i)} and the fact that $|\mathcal{Q}|\leq k^n$ we have that
$$f'(n, T_{k+1}, \eta)\leq k^n 3^{t_k(n)}C2^{-\eta n}.$$
Recall from~(\ref{modified rough structure equation}) that
$$f(n, T_{k+1})\leq f(n, T_{k+1}, \eta)(1+2^{-\varepsilon n^2}).$$
Together with the fact that $f(n, T_{k+1}, \eta)=f'(n, T_{k+1}, \eta)+T(n,k)$ and the upper bound in Lemma~\ref{bounds}{\rm (i)}, this implies that
\begin{align*}
f(n, T_{k+1})-T(n,k)&\leq f'(n, T_{k+1}, \eta)+ f(n, T_{k+1}, \eta)2^{-\varepsilon n^2}\\
&= f'(n, T_{k+1}, \eta)+ (f'(n, T_{k+1}, \eta)+T(n,k))2^{-\varepsilon n^2}\\
&\leq  k^n 3^{t_k(n)}C2^{-\eta n}+ k^n3^{t_k(n)}(1+ C2^{-\eta n})2^{-\varepsilon n^2}.
\end{align*}
Now the lower bound in Lemma~\ref{bounds}{\rm (i)} gives us that $f(n, T_{k+1})-T(n,k)=o(T(n,k))$. So  $f(n, T_{k+1})= T(n,k)(1+ o(1))$, as required.
\endproof

\section{Rough structure of typical $C_k$-free oriented graphs and digraphs}\label{Section: Rough structure of $C_k$-free oriented graphs and digraphs}

In this section we prove several stability results for $C_k$-free digraphs (Lemmas~\ref{cycle stability},~\ref{cycle stability digraphs} and~\ref{odd cycle stability}). These are used (together with Theorem~\ref{thm:contain}) at the end of the section to determine the `rough' structure of typical $C_k$-free oriented graphs and digraphs.

We will make use of the following definitions. For disjoint sets of vertices $A, A'$, we define $\overrightarrow{K}(A, A')$ to be the oriented graph on vertex set $A\cup A'$ with edge set consisting of all the $|A||A'|$ edges that are directed from $A$ to $A'$. Given a digraph $G$, $A\subseteq V(G)$ and $x\in V(G)\setminus A$, we say that $G[A, \{x\}]$ is an \emph{in-star} if $G[A, \{x\}]=\overrightarrow{K}(A, \{x\})$, and we say that $G[A, \{x\}]$ is an \emph{out-star} if $G[A, \{x\}]=\overrightarrow{K}(\{x\}, A)$. The following proposition will prove useful to us many times in this section.

\begin{proposition}\label{cycle proposition}
Let $a\in \mathbb{R}$ with $1\leq a\leq 2$, let $k\in \mathbb{N}$ with $k\geq 2$ and let $G$ be a $C_{k+1}$-free digraph. Suppose $G$ contains a copy $C$ of $C_k$ with vertex set $A\subseteq V(G)$, and let $x\in V(G)\setminus A$. Then the following hold:
\begin{enumerate}[{\rm (i)}]
\item $e_a(A, \{x\})\leq k$,
\item if $G[A, \{x\}]$ is not an in-star or an out-star, then $e_a(A, \{x\})\leq k-2+a$,
\item if $G[A, \{x\}]$ is not an in-star or an out-star, and contains no double edges, then $e_a(A, \{x\})\leq k-1$.
\end{enumerate}
Suppose moreover that for some $\ell \in \{k-1, k\}$, $G$ contains a copy $C'$ of $C_{\ell}$ with vertex set $A'\subseteq V(G)$, where $A\cap A'=\emptyset$. Then the following hold:
\begin{enumerate}[{\rm (i)}]
\setcounter{enumi}{3}
\item $e_a(A, A')\leq k\ell$,
\item if $G[A, A']\notin \{ \overrightarrow{K}(A, A'), \overrightarrow{K}(A', A) \}$, then $e_a(A, A')\leq k\ell-2+a$,
\item if $G[A, A']\notin \{ \overrightarrow{K}(A, A'), \overrightarrow{K}(A', A) \}$, and moreover $G[A, A']$ contains no double edges, then $e_a(A, A')\leq k\ell-1$.
\end{enumerate}
\end{proposition}
\begin{proof}
Write $C=v_1 v_2\dots v_k$. For $i\in [k]$ let $Q_i:= \{v_i x, x v_{i+1}\}$, where $v_{k+1}:= v_1$. Since $G$ is $C_{k+1}$-free we have that $|E(G)\cap Q_i|\leq 1$ for every $i\in [k]$. Hence $e(A, \{x\})\leq k$. We can now prove {\rm (i)--(vi)}.

\begin{enumerate}[{\rm (i)}]

\item This follows since $e_a(A, \{x\})\leq e(A, \{x\})\leq k$.

\item Suppose that $G[A, \{x\}]$ is not an in-star or an out-star. Note that if for some $j\in [k]$ we have that $E(G)\cap Q_j=\emptyset$ then, since $|E(G)\cap Q_i|\leq 1$ for all $i\in [k]$, $e_a(A, \{x\})\leq e(A, \{x\})\leq k-1\leq k-2+a$ as required. So we may assume that $|E(G)\cap Q_i|=1$ for every $i\in [k]$. Since $G[A, \{x\}]$ is not an in-star or an out-star, there exists some $j\in [k]$ such that $E(G)\cap Q_j=\{xv_{j+1}\}$ and $E(G)\cap Q_{j+1}=\{v_{j+1}x\}$; that is, there exists a double edge in $G[A, \{x\}]$. So since $e(G[A, \{x\}])\leq k$ we have that $e_a(A, \{x\})\leq k-2+a$, as required.

\item Suppose that $G[A, \{x\}]$ is not an in-star or an out-star, and contains no double edges. Just as in the proof of {\rm (ii)}, we have that if $|E(G)\cap Q_i|=1$ for every $i\in [k]$ then there exists a double edge in $G[A, \{x\}]$. So we may assume that for some $j\in [k]$ we have that $E(G)\cap Q_j=\emptyset$. This implies that $e_a(A, \{x\})\leq e(A, \{x\})\leq k-1$, as required.

\item This is immediate from {\rm (i)}.

\item Suppose that $G[A, A']\notin \{ \overrightarrow{K}(A, A'), \overrightarrow{K}(A', A) \}$. We claim that there exists $x\in A'$ such that $G[A, \{x\}]$ is not an in-star or an out-star. Indeed, suppose not. Then since $G[A, A']\notin \{ \overrightarrow{K}(A, A'), \overrightarrow{K}(A', A) \}$, there must exist distinct vertices $y',z'\in A'$ such that $G[A, \{y'\}]$ is an in-star and $G[A, \{z'\}]$ is an out-star. Let $P'$ be the subpath of $C'$ from $y'$ to $z'$. Thus $P'$ has length $s$ for some $1\leq s\leq k-1$. Let $y, z\in A$ be not necessarily distinct vertices such that the subpath $P$ of $C$ from $y$ to $z$ has length $k-s-1$. Then $yPzy'P'z'y$ is a copy of $C_{k+1}$ in $G$, which contradicts the assumption that $G$ is $C_{k+1}$-free. Hence there does exist $x\in A$ such that $G[A, \{x\}]$ is not an in-star or an out-star. So by {\rm (i)} and {\rm (ii)} we have that $e_a(A, A')\leq k\ell-2+a$ as required.

\item This proof is almost identical to that of {\rm (v)}, just using {\rm (iii)} instead of {\rm (ii)}, and so is omitted.
\end{enumerate}
\end{proof}

For $k\in \mathbb{N}$ define $T_{n,k}^+$ (up to isomorphism) to be the digraph on vertex set $[n]$ with all edges ${ij}$ where $i<j$ and all edges ${ji}$ where $i<j$ and $\left\lfloor (i-1)/k \right\rfloor = \left\lfloor (j-1)/k \right\rfloor$. So if $n=sk$ for some $s\in \mathbb{N}$ then $T_{n,k}^+$ is obtained from $T_s$ by blowing up each vertex to a copy of the complete digraph $DK_k$. Note that $T_{n,k}^+$ is $C_{k+1}$-free, and for all $a\in \mathbb{R}$ with $1\leq a\leq 2$ we have that
\begin{equation}\label{extremal number}
e_a(T_{n,k}^+)= \binom{n}{2}+\left\lfloor \frac{n}{k} \right\rfloor \binom{k}{2}(a-1)+ \binom{n-k\left\lfloor \frac{n}{k} \right\rfloor}{2}(a-1).
\end{equation}
We will first show that $T_{n,k}^+$ is an extremal digraph for $C_{k+1}$. The resulting formula for ${\rm ex}_a(n, C_{k+1})$ will be used in the proofs of Lemmas~\ref{cycle stability} and~\ref{cycle stability digraphs}, but we will not refer to $T_{n,k}^+$ itself again. Note that the case $a=2$ of Lemma~\ref{cycle extremal} corresponds to finding the digraph Tur\'an number of $C_{k+1}$, and is already due to H\"aggkvist and Thomassen~\cite{HaTh}.

\begin{lemma}\label{cycle extremal}
Let $a\in \mathbb{R}$ with $1\leq a\leq 2$ and let $k\in \mathbb{N}$. Then
$${\rm ex}_a(n, C_{k+1})=e_a(T_{n,k}^+).$$
%and $T_{n,k}^+, T_{n,k}^-$ are the only extremal digraphs.
\end{lemma}
\begin{proof}
The proof proceeds by induction on $k$. The case $k=1$ is trivial. So suppose $G$ is a $C_{k+1}$-free digraph on $n$ vertices for some $k>1$. The proof now proceeds by induction on $n$. The cases $n=1,\dots, k$ are trivial. So suppose $n>k$. Note that if $G$ is also $C_k$-free then by our inductive hypothesis on $k$,
$$e_a(G)\leq e_a(T_{n,k-1}^+) \leq e_a(T_{n,k}^+).$$

Otherwise, $G$ contains a copy of $C_k$, say on vertex set $A\subseteq V(G)$. So by Proposition~\ref{cycle proposition}{\rm (i)} we have that $e_a(A, \{x\})\leq k$ for every $x\in V(G)\setminus A$. Hence $e_a(A, V(G)\setminus A)\leq k(n-k)$. Note that by our inductive hypothesis on $n$, $e_a (G[V(G)\setminus A])\leq e_a(T^+_{n-k,k}).$ Hence,
\begin{eqnarray}
e_a(G) \hspace{-0.3cm} & = & \hspace{-0.3cm} e_a(G[V(G)\setminus A])+e_a(A, V(G)\setminus A)+e_a(G[A])\nonumber\\
& \stackrel{(\ref{extremal number})}{ \leq} & \hspace{-0.3cm} \binom{n-k}{2}+\left( \left\lfloor \frac{n-k}{k} \right\rfloor \binom{k}{2}+ \binom{(n-k)-k\left\lfloor \frac{n-k}{k} \right\rfloor}{2} \right) (a-1) + k(n-k) + a\binom{k}{2}\nonumber\\
& = & \hspace{-0.3cm} \binom{n}{2}+\left\lfloor \frac{n}{k} \right\rfloor \binom{k}{2}(a-1)+ \binom{n-k\left\lfloor \frac{n}{k} \right\rfloor}{2}(a-1)=e_a(T_{n,k}^+).\nonumber
\end{eqnarray}
So indeed ${\rm ex}_a(n, C_{k+1})=e_a(T_{n,k}^+),$ as required.
\end{proof}

We will next prove three stability results for $C_{k+1}$-free digraphs. The first will cover the case $1\leq a<2$ (where $a$, as usual, is the parameter in the definition of the weighted size of a digraph) and will be used to prove a structural result on $C_{k+1}$-free oriented graphs. The second covers the case $a=2$ and $k$ odd, and will be used to prove an analogous structural result on $C_{k+1}$-free digraphs for odd $k$. The third covers the case $a=2$ and $k$ even, and will be used to prove a (less restrictive) structural result on $C_{k+1}$-free digraphs for even $k$. The proofs of the first two of these stability results will make use of a result of Chudnovsky, Seymour and Sullivan~\cite{CSS}, which we state below. To do so we need to introduce the following notation. Let $\beta(G)$ denote the number of backwards edges in $G$ with respect to a transitive-optimal ordering of $G$. Let $\gamma(G)$ denote the number of unordered non-adjacent pairs of vertices in $G$; that is unordered pairs $u,v$ of vertices such that ${uv}\notin E(G)$ and ${vu}\notin E(G)$.

\begin{lemma}\cite{CSS}\label{CSS theorem}
Let $G$ be a $\{C_2,C_3\}$-free digraph. Then $\beta(G)\leq \gamma(G)$.
\end{lemma}

It is conjectured in~\cite{CSS} that in fact $\beta(G)\leq \gamma(G)/2$ for all $\{C_2,C_3\}$-free digraphs $G$. If true, this would be best possible.

\begin{lemma}[Stability when $a<2$]\label{cycle stability}
Let $a\in \mathbb{R}$ with $1\leq a< 2$ and let $k\in \mathbb{N}$ with $k\geq 2$. Then for all $\varepsilon>0$ there exist $\delta>0$ and $n_0\in \mathbb{N}$ such that every $C_{k+1}$-free digraph $G$ on $n\geq n_0$ vertices with
$$e_a(G)\geq \binom{n}{2}-\delta n^2$$
satisfies $G=T_n \pm \varepsilon n^2$.
\end{lemma}
\begin{proof}
We prove the lemma via the following claim.

\vspace{0.2cm}
\noindent{\bf Claim:} \textit{Let} $k\in \mathbb{N}$ \textit{with} $k\geq 2$ \textit{and let} $\varepsilon>0$. \textit{Suppose that there exist} $\delta'>0$ \textit{and} $n_0'\in \mathbb{N}$ \textit{such that every} $\{C_k, C_{k+1}\}$-\textit{free digraph} $G$ \textit{on} $n'\geq n_0'$ \textit{vertices with}
$$e_a(G)\geq \binom{n'}{2}-\delta' n'^2$$
\textit{satisfies} $G=T_{n'} \pm \varepsilon n'^2/(2k^2)$. \textit{Then there exist} $\delta>0$ \textit{and} $n_0\in \mathbb{N}$ \textit{such that every} $C_{k+1}$-\textit{free digraph} $G$ \textit{on} $n\geq n_0$ \textit{vertices with}
$$e_a(G)\geq \binom{n}{2}-\delta n^2$$
\textit{satisfies} $G=T_n \pm \varepsilon n^2$.

\vspace{0.2cm}
In order to check that the claim implies the lemma, we proceed by induction on $k$. For the base case $k=2$ the assumption of the claim is satisfied, since if $\delta':= \varepsilon/(4k^2)$ and if $G$ is a $\{C_2, C_{3}\}$-free digraph on $n'$ vertices with $e_a(G)\geq \binom{n'}{2}-\delta' n'^2$ then $\gamma(G)\leq \delta' n'^2$, and so applying Lemma~\ref{CSS theorem} yields the assumption of the claim. So the conclusion of the claim holds, which is precisely the statement of the lemma for $k=2$. For $k>2$ the assumption of the claim is satisfied by the inductive hypothesis (since any $\{C_k, C_{k+1}\}$-free digraph is certainly a $C_k$-free digraph) and so the conclusion of the claim holds, which is precisely the statement of the lemma for $k$. So by induction the lemma holds and we are done.

Thus it remains to prove the claim. (Note that, apart from in the base case $k=2$, in the above argument it would suffice for the assumption in the statement of the claim to refer to $C_k$-free digraphs, rather than $\{C_k, C_{k+1}\}$-free digraphs. As such, this claim is stronger than strictly necessary for our purposes, since the assumption is weaker. However, this approach allows us to prove the base case at the same time as the inductive step, and so yields a shorter proof.)

\vspace{0.2cm}
\noindent{\bf Proof of claim:} Choose $\delta$ and $n_0$ such that $1/n_0 \ll \delta \ll 1/k,2-a, \delta'$ and $1/n_0 \ll 1/n_0', \varepsilon$. Let $G$ be a $C_{k+1}$-free digraph on $n\geq n_0$ vertices with
\begin{equation}\label{tilde edges lower bound}
e_a(G)\geq \binom{n}{2}-\delta n^2.
\end{equation}

Let $t\geq 0$ denote the maximum number of vertex-disjoint copies of $C_k$ in $G$. Let $\mathcal{C}=\{C^{1},\dots, C^{t} \}$ be a set of $t$ vertex-disjoint copies of $C_k$ in $G$. Let $V_1:= V(C^{1})\cup \dots \cup V(C^{t})$ and $V_2:= V(G)\setminus V_1$. Let $n_1:= |V_1|$ and $n_2:= |V_2|$. Note that $G[V_2]$ is $C_k$-free.

Note that Proposition~\ref{cycle proposition}{\rm (i)} implies that $e_a(V_1, V_2)\leq n_1 n_2$, since $G$ is $C_{k+1}$-free. Also, for $i=1,2$, since $G[V_i]$ is $C_{k+1}$-free, Lemma~\ref{cycle extremal} and (\ref{extremal number}) together imply that
$$e_a(G[V_i])\leq {\rm ex}_a(n_i, C_{k+1}) =\binom{n_i}{2}+\left\lfloor \frac{n_i}{k} \right\rfloor \binom{k}{2}(a-1) + \binom{n_i-k\left\lfloor \frac{n_i}{k} \right\rfloor}{2}(a-1) \leq \binom{n_i}{2}+\delta n^2.$$
(The last inequality holds since $1/n_0 \ll \delta \ll 1/k$.) Together with (\ref{tilde edges lower bound}) this implies that
\begin{align}\label{stability equation 1}
e_a(G[V_1])&= e_a(G) - e_a(V_1, V_2) - e_a(G[V_2])\\
&\geq \left( \binom{n}{2} - \delta n^2 \right) - n_1 n_2 - \left( \binom{n_2}{2} + \delta n^2 \right) = \binom{n_1}{2} - 2 \delta n^2,\nonumber
\end{align}
and similarly that\COMMENT{\begin{align*}
e_a(G[V_2])&= e_a(G) - e_a(V_1, V_2) - e_a(G[V_1]) \geq \left( \binom{n}{2} - \delta n^2 \right) - n_1 n_2 - \left( \binom{n_1}{2} + \delta n^2 \right)\\
&= \binom{n_2}{2} - 2 \delta n^2,
\end{align*}}
\begin{equation}\label{stability equation 2}
e_a(G[V_2])= e_a(G) - e_a(V_1, V_2) - e_a(G[V_1]) \geq \binom{n_2}{2} - 2 \delta n^2,
\end{equation}
and that\COMMENT{\begin{align*}
e_a(V_1, V_2)&= e_a(G) - e_a(G[V_1]) - e_a(G[V_2]) \\
&\geq \left( \binom{n}{2} - \delta n^2 \right) - \left( \binom{n_1}{2} + \delta n^2 \right) - \left( \binom{n_2}{2} + \delta n^2 \right)\\
&= n_1 n_2 - 3 \delta n^2.
\end{align*}}
\begin{equation}\label{stability equation 3}
e_a(V_1, V_2)= e_a(G) - e_a(G[V_1]) - e_a(G[V_2]) \geq n_1 n_2 - 3 \delta n^2.
\end{equation}

We now consider the digraph $G'$ defined on vertex set $[t]\cup V_2$ as follows. Firstly, $G'[V_2]:=G[V_2]$. For vertices $i,j\in [t]$ we have that ${ij}\in E(G')$ if and only if $G[V(C^i), V(C^j)] = \overrightarrow{K}(V(C^i), V(C^j))$. For a vertex $x\in V_2$ and an element $i\in [t]$ we have that ${ix}\in E(G')$ if and only if $G[V(C^{i}), \{x\}]$ is an in-star and that ${xi}\in E(G')$ if and only if $G[V(C^{i}), \{x\}]$ is an out-star.

Note that by Proposition~\ref{cycle proposition}(iv), $e_a(G[V(C^i), V(C^j)])\leq k^2$ (for all $i\ne j$). Moreover, Proposition~\ref{cycle proposition}(v) implies that if $i,j\in [t]$ and ${ij}, {ji}\notin E(G'[[t]])$ then $e_a(G[V(C^{i}), V(C^{j})])\leq k^2-2+a$. Let $s:= \binom{t}{2}-e_a(G'[[t]])=\binom{t}{2}-e(G'[[t]])$. Then
\begin{eqnarray}\label{reference align}
\binom{n_1}{2}-2\delta n^2 \hspace{-0.3cm} & \stackrel{(\ref{stability equation 1})}{\leq} & \hspace{-0.3cm} e_a(G[V_1])=\sum\limits_{\substack{i,j\in [t]\\ i<j}} e_a(G[V(C^i), V(C^j)]) + \sum\limits_{i\in [t]} e_a(G[V(C^i)])\\
& \leq & \hspace{-0.3cm} \binom{n_1}{2} - s(2-a) + t\binom{k}{2}(a-1)\nonumber.
\end{eqnarray}
Thus $s(2-a)\leq 3\delta n^2$, i.e.
$$e_a(G'[[t]])\geq \binom{t}{2}- \frac{3\delta n^2} {2-a}.$$
Similarly, Proposition~\ref{cycle proposition}(ii) implies that if $i\in [t]$ and $x\in V_2$ and ${ix}, {xi}\notin E(G'[[t]])$ then $e_a(G[V(C^{i}), \{x\}])\leq k-2+a$. So we have that
$$e_a(G'[[t], V_2]) \stackrel{(\ref{stability equation 3})}{\geq} tn_2 - \frac{3\delta n^2}{2-a}.$$
So recalling that $e_a(G'[V_2]) = e_a(G[V_2])$ we have that
\begin{equation}\label{te G' equation}
e_a(G')=e_a(G'[[t]]) + e_a(G'[V_2]) + e_a(G'[[t], V_2]) \stackrel{(\ref{stability equation 2})}{\geq} \binom{t+n_2}{2} - \frac{8\delta n^2}{2-a}.
\end{equation}
Since $t+n_2\geq n/k$ we have that $t+n_2\geq n_0'$ and that $8\delta n^2/(2-a) \leq \delta' (t+n_2)^2$, and hence by (\ref{te G' equation}) that
\begin{equation}\label{te G' equation 2}
e_a(G')\geq \binom{t+n_2}{2} - \delta' (t+n_2)^2.
\end{equation}

We now claim that $G'$ must be $\{C_k, C_{k+1}\}$-free. Indeed, suppose not. If $G'$ contains a copy of $C_{k+1}$ then it is clear that $G$ also contains a copy of $C_{k+1}$, contradicting our assumption. So we may assume that $G'$ contains a copy of $C_k$. Since $G'[V_2]=G[V_2]$ is $C_k$-free by construction, the vertex set of any copy of $C_k$ in $G'$ must contain some $i\in[t]$. But then $G$ would clearly contain a copy of $C_{k+1}$ using two of the vertices in $V(C^{i})$, again contradicting our assumption that $G$ is $C_{k+1}$-free. So $G'$ is $\{C_k, C_{k+1}\}$-free, as claimed.

Thus by~(\ref{te G' equation 2}) and the assumption in the statement of the claim we have that\COMMENT{Note that we have the factor of $2$ in the denominator of the error term here to account for the fact that, when going from the number of `bad' edges in $G'$ to the number in $G$, we have no control over the edges inside the vertex sets $V(C^1),\dots, V(C^t)$, so we have to assume that they could all be `bad' too.} $G'=T_{t+n_2} \pm \varepsilon (t+n_2)^2/(2k^2)$. Together with the definition of $G'$ this implies that $G = T_n \pm \varepsilon n^2$, as required. This completes the proof of the claim, and hence completes the proof of the lemma.
\end{proof}

The rough strategy of the next proof is similar to that of Lemma~\ref{cycle stability}.

\begin{lemma}[Stability when $a=2$ and $k$ is odd]\label{cycle stability digraphs}
Let $k\in \mathbb{N}$ with $k\geq 3$ and $k$ odd. Then for all $\varepsilon>0$ there exist $\delta>0$ and $n_0\in \mathbb{N}$ such that every $C_{k+1}$-free digraph $G$ on $n\geq n_0$ vertices with
$$e(G)\geq \binom{n}{2}-\delta n^2$$
satisfies $G=T_n \pm \varepsilon n^2$.
\end{lemma}
\begin{proof}
We prove the lemma via the following claim.

\vspace{0.2cm}
\noindent{\bf Claim:} \textit{Let} $k\in \mathbb{N}$ \textit{with} $k\geq 3$ \textit{and} $k$ \textit{odd, and let} $\varepsilon>0$. \textit{Suppose that there exist} $\delta'>0$ \textit{and} $n_0'\in \mathbb{N}$ \textit{such that every} $\{C_{k-1}, C_k\}$-\textit{free digraph} $G$ \textit{on} $n'\geq n_0'$ \textit{vertices with}
$$e(G)\geq \binom{n'}{2}-\delta' n'^2$$
\textit{satisfies} $G=T_{n'} \pm \varepsilon n'^2/(2k^2)$. \textit{Then there exist} $\delta>0$ \textit{and} $n_0\in \mathbb{N}$ \textit{such that every} $C_{k+1}$-\textit{free digraph} $G$ \textit{on} $n\geq n_0$ \textit{vertices with}
$$e(G)\geq \binom{n}{2}-\delta n^2$$
\textit{satisfies} $G=T_n \pm \varepsilon n^2$.

\vspace{0.2cm}
In order to check that the claim implies the lemma, we proceed by induction on $\ell:=(k+1)/2$. The argument is similar to that in the proof of Lemma~\ref{cycle stability}. (As before, Lemma~\ref{CSS theorem} implies that in the base case $\ell=2$ of the induction, the assumption of the claim holds.)\COMMENT{For the base case $\ell=2$ the assumption of the claim is satisfied, since if $G$ is a $\{C_2, C_{3}\}$-free digraph on $n'$ vertices with $e(G)\geq \binom{n'}{2}-\delta' n'^2$ then $\gamma(G)\leq \delta' n'^2$, and so applying Lemma~\ref{CSS theorem} yields the result if $\delta'$ is sufficiently small. So the conclusion of the claim holds, which is precisely the statement of the lemma for $k=3$. For $\ell>2$ the assumption of the claim is satisfied by the inductive hypothesis (since any $\{C_{k-1}, C_k\}$-free digraph is certainly a $C_{k-1}$-free digraph) and so the conclusion of the claim holds, which is precisely the statement of the lemma for $k$. So by induction the lemma holds and we are done.
Thus it remains to prove the claim. (Note that, similarly to in the proof of Lemma~\ref{cycle stability}, this claim is stronger than necessary, but allows for a shorter proof method than otherwise.)}

\vspace{0.2cm}
\noindent{\bf Proof of claim:} Choose $\delta$ and $n_0$ such that $1/n_0 \ll \delta \ll 1/k, \delta'$ and $1/n_0 \ll 1/n_0', \varepsilon$. Let $G$ be a $C_{k+1}$-free digraph on $n\geq n_0$ vertices with
\begin{equation}\label{tilde edges lower bound digraphs}
e(G)\geq \binom{n}{2}-\delta n^2.
\end{equation}

Let $t\geq 0$ denote the maximum number of vertex-disjoint copies of $C_k$ in $G$. Let $\mathcal{C}=\{C^{1},\dots, C^{t} \}$ be a set of $t$ vertex-disjoint copies of $C_k$ in $G$. Let $V_1:= V(C^{1})\cup \dots \cup V(C^{t})$ and $n_1:= |V_1|$. Now let $t^*\geq 0$ denote the maximum number of vertex-disjoint copies of $C_{k-1}$ in $G[V\setminus V_1]$. Let $\mathcal{C}^*=\{C^{1}_*,\dots, C^{t^*}_* \}$ be a set of $t^*$ vertex-disjoint copies of $C_{k-1}$ in $G[V\setminus V_1]$. Let $V_2:= V(C^{1}_*)\cup \dots \cup V(C^{t^*}_*)$ and $n_2:= |V_2|$. Let $V_3:= V(G)\setminus (V_1 \cup V_2)$ and $n_3:= |V_3|$. Note that $G[V_2\cup V_3]$ is $C_k$-free and that $G[V_3]$ is $\{C_{k-1}, C_k\}$-free.

Proposition~\ref{cycle proposition}{\rm (i)} implies that $e(V_1, V_2)\leq n_1 n_2$ and that $e(V_1, V_3)\leq n_1 n_3$, since $G$ is $C_{k+1}$-free, and that $e(V_2, V_3)\leq n_2 n_3$, since $G[V_2\cup V_3]$ is $C_{k}$-free. Also, similarly to the proof of Lemma~\ref{cycle stability} we use Lemma~\ref{cycle extremal} to get that $e(G[V_i])\leq \binom{n_i}{2}+\delta n^2$ for $i=1,2,3$. Together with (\ref{tilde edges lower bound digraphs}) this implies that
\begin{align*}
e(G[V_1])&= e(G) - e(V_1, V_2) - e(V_1, V_3) - e(V_2, V_3) - e(G[V_2]) - e(G[V_3])\\
&\geq \left( \binom{n}{2} - \delta n^2 \right) - n_1 n_2 - n_1 n_3 - n_2 n_3 - \left( \binom{n_2}{2} + \delta n^2 \right) - \left( \binom{n_3}{2} + \delta n^2 \right)\\
&= \binom{n_1}{2} - 3 \delta n^2,
\end{align*}
and similarly that $e(G[V_2]) \geq \binom{n_2}{2} - 3 \delta n^2$, and that $e(G[V_3]) \geq \binom{n_3}{2} - 3 \delta n^2,$ and that $e(V_i, V_j)\geq n_i n_j - 4 \delta n^2$ for all $i, j\in \{1,2,3\}$, $i\ne j$.

We now consider the digraph $G'$ defined on vertex set $([t] \times \{0\})\cup ([t^*] \times \{1\})\cup V_3$ as follows. Firstly, for every vertex $v\in V(G')$ define $f(v)$ to be $\{v\}$ if $v\in V_3$, to be $V(C^{i})$ if $v=(i, 0)\in [t]\times \{0\}$, and to be $V(C^{i}_*)$ if $v=(i, 1)\in [t^*] \times \{1\}$. Now let $G'[V_3]:= G[V_3]$ and for vertices $u, v\in V(G')$ with $|\{u,v\}\cap V_3|\leq 1$ define $uv\in E(G')$ if and only if $G[f(u), f(v)] = \overrightarrow{K}(f(u), f(v))$.

Note that by the K\H{o}v\'ari-S\'os-Tur\'an theorem, $G$ contains at most $\delta n^2$ double edges, since $G$ is $C_{k+1}$-free and $k+1$ is even, and $1/n_0 \ll \delta \ll 1/k$. Note also that by Proposition~\ref{cycle proposition}{\rm (iv)}, if $u,v\in [t] \times \{0\}$ then $e(G[f(u), f(v)])\leq k^2$. If, moreover, $uv, vu \notin E(G'[[t] \times \{0\}])$ and $G[f(u), f(v)]$ contains no double edge, then by Proposition~\ref{cycle proposition}{\rm (vi)} we have that $e(G[f(u), f(v)])\leq k^2-1$. Using that $e(G[V_1])\geq \binom{n_1}{2} - 3 \delta n^2$ one can now argue similarly as in~(\ref{reference align}) to see that\COMMENT{Let $s:= \binom{t}{2} - e(G'[[t] \times \{0\}])$. Then
\begin{align*}
\binom{n_1}{2} - 3\delta n^2 & \leq e(G[V_1]) = \sum\limits_{\substack{i,j\in [t]\\ i<j}} e(G[V(C^i), V(C^j)]) + \sum\limits_{i\in [t]} e(G[V(C^i)])\\
&\leq \binom{n_1}{2} - (s-\delta n^2) + t\binom{k}{2},
\end{align*}
and hence $s\leq 5\delta n^2$, as required. The $(s-\delta n^2)$ term above is due to the fact that, of the $s$ pairs of indices $i,j$ that correspond to non-edges in $G'[[t] \times \{0\}]$, at most $\delta n^2$ of them are such that $G[V(C^i), V(C^j)]$ contains a double edge, since we observed that $G$ contains at most $\delta n^2$ double edges (note that in this case we still have that $e(G[V(C^i), V(C^j)])\leq k^2$), and so there are at least $(s-\delta n^2)$ pairs remaining that are such that $e(G[V(C^i), V(C^j)])\leq k^2-1$.}
$$e(G'[[t] \times \{0\}]) \geq \binom{t}{2} - 5\delta n^2.$$
Also, Proposition~\ref{cycle proposition}{\rm (iv)} implies that if $u\in [t] \times \{0\}$ and $v\in [t^*] \times \{1\}$ then $e(G[f(u), f(v)])\leq k(k-1)$. If, moreover, ${uv}, {vu}\notin E(G'[[t] \times \{0\}, [t^*] \times \{1\}])$ and $G[f(u), f(v)]$ contains no double edge, then by Proposition~\ref{cycle proposition}{\rm (vi)} we have that $e(G[f(u), f(v)])\leq k(k-1)-1$. Using that $e(V_1, V_2)\geq n_1 n_2 - 4 \delta n^2$ one can again argue similarly as in~(\ref{reference align}) to see that
$$e(G'[[t] \times \{0\}, [t^*] \times \{1\}]) \geq t t^* - 5\delta n^2.$$
Furthermore, Proposition~\ref{cycle proposition}{\rm (i)} implies that if $u\in [t] \times \{0\}$ and $v\in V_3$ then $e(G[f(u), f(v)])\leq k$. If, moreover, ${uv}, {vu}\notin E(G'[[t] \times \{0\}, V_3])$ and $G[f(u), f(v)]$ contains no double edge, then by Proposition~\ref{cycle proposition}{\rm (iii)} we have that $e(G[f(u), f(v)])\leq k-1$. Using that $e(V_1, V_3)\geq n_1 n_3 - 4 \delta n^2$ one can again argue similarly as in~(\ref{reference align}) to see that
$$e(G'[[t] \times \{0\}, V_3]) \geq t n_3 - 5\delta n^2.$$

Using that $G[V_2\cup V_3]$ is $C_k$-free, in a similar way we get that $e(G'[[t^*] \times \{1\}])\geq \binom{t^*}{2} - 5\delta n^2$ and that $e(G'[[t^*] \times \{1\}, V_3])\geq t^* n_3 - 5\delta n^2$. So recalling that $e(G'[V_3]) = e(G[V_3])\geq \binom{n_3}{2} - 3 \delta n^2$ we have that
\begin{align}\label{e_a(G') bound}
e(G')= &\hspace{0.1cm} e(G'[[t] \times \{0\}]) + e(G'[[t^*] \times \{1\}]) + e(G'[V_3])\\
&+ e(G'[[t] \times \{0\}, [t^*] \times \{1\}]) + e(G'[[t] \times \{0\}, V_3]) + e(G'[[t^*] \times \{1\}], V_3)\nonumber\\
\geq &\hspace{0.05cm} \binom{t+t^*+n_3}{2} - 28 \delta n^2\nonumber.
\end{align}

Since $t + t^* +n_3\geq n/k$, we have that $t + t^* +n_3\geq n_0'$ and that $28 \delta n^2 \leq \delta' (t + t^* +n_3)^2$, and hence by (\ref{e_a(G') bound}) that
\begin{equation}\label{te G' equation digraphs 2}
e(G')\geq \binom{t + t^* +n_3}{2} - \delta' (t + t^* +n_3)^2.
\end{equation}

We now claim that $G'$ must be $\{C_{k-1}, C_k\}$-free. Indeed, suppose not. If $G'$ contains a copy of $C_k$ then since $G'[V_3]=G[V_3]$ is $C_k$-free by construction, the vertex set of such a copy of $C_k$ in $G'$ must contain some $u\in([t] \times \{0\})\cup ([t^*] \times \{1\})$. But then $G$ would clearly contain a copy of $C_{k+1}$ using two of the vertices in $f(u)$, contradicting our assumption that $G$ is $C_{k+1}$-free. Similarly, if $G'$ contains a copy of $C_{k-1}$ then since $G[V_3]$ is $C_{k-1}$-free by construction, the vertex set of such a copy of $C_{k-1}$ in $G'$ must contain some $u\in([t] \times \{0\})\cup ([t^*] \times \{1\})$. If there exists such a $u\in [t] \times \{0\}$ then $G$ would clearly contain a copy of $C_{k+1}$ using three of the vertices in $f(u)$, since $|f(u)|=k\geq 3$. Otherwise, there exists $u\in [t^*] \times \{1\}$ and a copy of $C_{k-1}$ in $G'$ that uses $u$ but no vertices in $[t] \times \{0\}$. But then $G[V_2\cup V_3]$ would clearly contain a copy of $C_k$ using two of the vertices in $f(u)$, contradicting our previous observation that $G[V_2\cup V_3]$ is $C_k$-free.

So $G'$ is $\{C_{k-1}, C_k\}$-free, as claimed. Thus by~(\ref{te G' equation digraphs 2}) and the assumption in the statement of the claim we have that $G'=T_{t + t^* +n_3} \pm \varepsilon (t + t^* +n_3)^2/(2k^2)$. Together with the definition of $G'$ this implies that $G = T_n \pm \varepsilon n^2$, as required. This completes the proof of the claim, and hence completes the proof of the lemma.
\end{proof}

We now prove a digraph stability result for forbidden odd cycles.
Here both the Tur\'an and the stability results allow for a richer structure than in the previous two lemmas:
For even $k$,
a near extremal graph can be obtained from a transitive tournament by blowing up some of its vertices into
complete bipartite graphs of arbitrary size (see Section~\ref{Section: Introduction} for the precise definition).
This makes the proof more difficult than the previous two.

\begin{lemma}[Stability when $a=2$ and $k$ is even]\label{odd cycle stability}
Let $k\in \mathbb{N}$ be even. Then for all $\varepsilon>0$ there exist $\delta>0$ and $n_0\in \mathbb{N}$ such that every $C_{k+1}$-free digraph $G$ on $n\geq n_0$ vertices with
\begin{equation}\label{G close to complete}
e(G)\geq \binom{n}{2}-\delta n^2
\end{equation}
can be made into a transitive-bipartite blow up by changing at most $\varepsilon n^2$ edges.
\end{lemma}

To give an idea of the proof, consider the triangle-free case $k=2$. In this case, we first consider a maximal collection $\mathcal{A}$ of disjoint double edges in $G$. 
It is easy to see that for almost all pairs of (double) edges $u_1u_2,v_1v_2 \in \mathcal{A}$,
either (i) $G[\{u_1,u_2,v_1,v_2\}]$ is a complete balanced bipartite digraph 
or (ii) $G$ contains all four possible edges from $\{u_1,u_2\}$ to $\{v_1,v_2\}$ (or vice versa).
We consider the following auxiliary `semi-oriented graph' $G'$  whose vertex set is $\mathcal{A}$.
In case (i), we include an (undirected) red edge between $u_1u_2$ and $v_1v_2$ in $G'$.
In case (ii), we include a blue edge directed from $u_1u_2$ to $v_1v_2$ in $G'$ (or vice versa).
One can now show  that the red edges induce a set of disjoint almost complete graphs $R$ in $G'$.
We then contract each such red almost complete graph $R$ into a vertex $v_R$ to obtain an oriented graph $J$
(vertices of $G'$ which are not involved in any of these $R$ are also retained in $J$).
So all edges of $J$ are blue.
Crucially, it turns out that $J$ is close to a transitive tournament.
Moreover, in $G$ each $v_R$ corresponds to an almost complete bipartite digraph, so altogether this shows 
that the subgraph of $G$ induced by the edges in $\mathcal{A}$ is close to a transitive-bipartite blow up. 
One can generalize this argument to incorporate the vertices of $G$ not covered by edges in $\mathcal{A}$
(these will only be incident to blue edges in $G'$ and $J$ and not to any red ones).

To formalize the above argument, we make use of the following definitions. A \textit{semi-oriented graph} is obtained from an undirected graph by first colouring each of the edges either red or blue and then giving an orientation to each of the blue edges. So a semi-oriented graph is a pair $G=(V,E)$, where $V$ is a set of vertices and $E$ is a set of coloured edges, some of which are red and undirected and the rest of which are blue and directed. We define basic notions such as induced subgraphs of $G$ in the obvious way.
Given a vertex $v\in V$ we denote the set of all vertices $x\in V$ for which there is a blue directed edge $vx\in E$ by $N^+_{G}(v)$. We call
the vertices in $N^+_{G}(v)$ \emph{blue out-neighbours of $v$}. We define the sets $N^-_{G}(v)$ of \emph{blue in-neighbours of $v$} and $N^{\text{red}}_{G}(v)$ of \emph{red neighbours of $v$} in a similar way.
If $x\in N^+_{G}(v)\cup N^-_{G}(v)$ we say that \emph{$x$ is a blue neighbour of $v$}.

We denote the complete bipartite digraph (with edges in both directions) with vertex classes of sizes $a$ and $b$ by $DK_{a,b}$.
The following simple proposition will also be used in the proof of Lemma~\ref{odd cycle stability}. We omit its straightforward proof, the details of which can be found in~\cite{Townsend PhD}.

\begin{proposition}\label{cycle proposition 2}
Let $k\in \mathbb{N}$ be even and let $G$ be a $C_{k+1}$-free digraph. Suppose $G$ contains a copy of $DK_{k/2, k/2}$ with vertex classes $A,B$. Then the following hold.
\begin{enumerate}[{\rm (i)}]
\item Suppose $x\in V(G)\setminus (A\cup B)$. Then $e(G[\{x\}, A\cup B])\leq k$, with equality only if $G[\{x\}, A\cup B]= \overrightarrow{K}(\{x\}, A\cup B)$ or $G[\{x\}, A\cup B]= \overrightarrow{K}(A\cup B, \{x\})$ or $G[\{x\}\cup A\cup B]=DK_{k/2+1, k/2}$.
\item Suppose $G$ contains another copy of $DK_{k/2, k/2}$ with vertex classes $C, D$ such that $(C\cup D) \cap (A\cup B)= \emptyset$. Then $e(G[C\cup D, A\cup B])\leq k^2$, with equality only if $G[C\cup D, A\cup B]= \overrightarrow{K}(C\cup D, A\cup B)$ or $G[C\cup D, A\cup B]= \overrightarrow{K}(A\cup B, C\cup D)$ or $G[C\cup D\cup A\cup B]=DK_{k, k}$.
\end{enumerate}
\end{proposition}

\removelastskip\penalty55\medskip\noindent{\bf Proof of Lemma~\ref{odd cycle stability}.}
Choose $n_0, \delta, \varepsilon_1, \varepsilon_2$ such that $1/n_0 \ll \delta \ll \varepsilon_1 \ll \varepsilon_2 \ll 1/k, \varepsilon$. Let $G$ be a $C_{k+1}$-free digraph on $n\geq n_0$ vertices which satisfies
(\ref{G close to complete}).

Let $t\geq 0$ denote the maximum number of vertex-disjoint copies of $DK_{k/2, k/2}$ in $G$. Let $\mathcal{A}=\{A^{1},\dots, A^{t} \}$ be
a collection of vertex sets of $t$ vertex-disjoint copies of $DK_{k/2, k/2}$ in $G$. Let $V_1:= A^{1}\cup \dots \cup A^{t}$ and let $V_2:= V(G)\setminus V_1$.
Note that $G[V_2]$ is $DK_{k/2, k/2}$-free, and hence by the K\H{o}v\'ari-S\'os-Tur\'an theorem $G[V_2]$ contains at most $\delta n^2$ double edges.

\vspace{0.2cm}
\noindent{\bf Claim 1:} \textit{For each $i\in [t]$ there are at most $6\delta^{1/2} n$ vertices $x\in V_2$ for which
$G[A^i\cup \{x\}]=DK_{k/2, k/2+1}$.}

\smallskip

\noindent
Indeed, suppose that there exists a set $X$ of more than $6\delta^{1/2} n$ such vertices. Proposition~\ref{cycle proposition 2}(ii) and
the fact that $1/n_0 \ll \delta,1/k$ together imply that $e(G[V_1])\le \binom{|V_1|}{2}+\delta n^2$. Moreover,
$e(V_1,V_2)\le |V_1||V_2|$ by Proposition~\ref{cycle proposition 2}(i). Together with our previous observation that
$G[V_2]$ contains at most $\delta n^2$ double edges and (\ref{G close to complete}) this implies that
$e(G[X])\ge \binom{|X|}{2}-3\delta n^2$. But this means that there are $x,y\in X$ such that $xy\in E(G)$ and
such that both $x$ and $y$ are joined with double edges to the same vertex
class of $G[A^i]=DK_{k/2, k/2}$, which contradicts the fact that $G$ is $C_{k+1}$-free.

\vspace{0.2cm}
Let $G^*$ be the digraph obtained from $G$ by deleting the at most $\delta n^2$ double edges in $G[V_2]$ and deleting the double
edges between $A^i$ and all the vertices $x\in V_2$ for which $G[A^i\cup \{x\}]=DK_{k/2, k/2+1}$ (for each $i\in [t]$).
By Claim~1, for each $i\in [t]$, the number of the latter double edges is at most $6\delta^{1/2}n\cdot k/2=3k\delta^{1/2} n$. Thus%
   \COMMENT{get at most $n/k\cdot 6k\delta^{1/2} n=3\delta^{1/2} n^2$ double edges of the 2nd kind, but have to multiply by 2, since each double
edge counts twice}
\begin{equation}\label{G close to complete2}
e(G^*)\ge e(G)- 2\left(\delta n^2 + \frac{n}{k}\cdot 3k \delta^{1/2} n\right)\stackrel{(\ref{G close to complete})}{\ge} \binom{n}{2}-7\delta^{1/2} n^2.
\end{equation}

Consider the semi-oriented graph $G'=(V',E')$ where $V':=\mathcal{A}\cup V_2$ and the edge set $E'$ is defined as follows. Firstly,
for every vertex $v\in V'$ define $f(v)$ to be $v$ if $v\in \mathcal{A}$, and to be $\{v\}$ if $v\in V_2$. If $u,v\in V'$ then there
is a blue edge in $E'$ directed from $u$ to $v$ if $G[f(u), f(v)]=\overrightarrow{K}(f(u), f(v))$.
If $u,v\in \mathcal{A}$ then there is a red edge in $E'$ between $u$ and $v$ if $G[f(u)\cup f(v)]=DK_{k,k}$. So $G'[V_2]=G^*[V_2]$.

Note that, since $G$ is $C_{k+1}$-free, $G'$ cannot contain any copy of $C_{k+1}$ which contains at least one blue edge and in which all the blue edges are oriented consistently
(as any such copy of $C_{k+1}$ in $G'$ would correspond to a $C_{k+1}$ in $G$).

Let $V'_0$ denote the set of all vertices $v\in V'$ for which there are at least $\delta^{1/4} n$ vertices $u\in V'$ such that $G'$ does not contain an edge between $v$ and $u$.
Note that Proposition~\ref{cycle proposition 2} together with (\ref{G close to complete2}) implies that
$|E'|\geq \binom{|V'|}{2} - 7\delta^{1/2} n^2.$
Hence\COMMENT{Note that we have a factor of $15$ instead of $14$ since for each $v\in V'_0$ there might only be
$\delta^{1/4} n-1$ vertices $u\in V'\setminus \{v\}$ such that $G'$ does not contain an edge between $v$ and $u$.}
$|V'_0|\leq 15\delta^{1/4} n$. Let $G'':= G'-V'_0$, $V'':=V'\setminus V'_0$ and $E'':=E(G'')$.

\vspace{0.2cm}
\noindent{\bf Claim 2:} \ 
\begin{itemize}
\item[(a)] \textit{For every vertex $v\in V''$ there are at most $\delta^{1/4} n$ vertices $u\in V''$
such that $G''$ does not contain an edge between $v$ and $u$.}
\item[(b)] \textit{$G''$ does not contain a triangle $uvw$ such that both $uv$ and $vw$ are red edges and $wu$ is a blue edge.}
\item[(c)] \textit{$G''$ does not contain a triangle $uvw$ such that $uv$ is a red edge and both $vw$ and $wu$ are (directed) blue edges.}
\end{itemize}

\smallskip

\noindent
Indeed, (a) is clear from the definition of $G''$, while (b) and (c) follow easily from the fact that $G$ is $C_{k+1}$-free.

\medskip

Given $q,q'\in\mathbb{N}$, we say that $U\subseteq V''$ is a \emph{red $(q,q')$-clique} if $|U|\geq q'$ and $|U\setminus N^\text{red}_{G''}(u)|\le q$ for all $u\in U$.

\vspace{0.2cm}
\noindent{\bf Claim 3:} \textit{Suppose that $R$ is a red $(\delta^{1/4} n,\eps_1 n)$-clique. Then the following hold.
\begin{itemize}
\item[{\rm (a)}] $G''[R]$ does not contain a blue edge.
\item[{\rm (b)}] No vertex $v\in V''\setminus R$ has both a red and a blue neighbour in $R$.
\item[{\rm (c)}] No vertex $v\in V''\setminus R$ has both a blue in-neighbour and a blue out-neighbour in $R$.
\end{itemize}
Suppose that $R'$ is another red $(\delta^{1/4} n,\eps_1 n)$-clique such that $R\cap R'=\emptyset$. Then the following hold.
\begin{itemize}
\item[{\rm (d)}] $G''$ cannot contain both a red edge and a blue edge between $R$ and $R'$.
\item[{\rm (e)}] $G''$ cannot contain both a directed blue edge from some vertex in $R$ to some vertex in $R'$ and a
directed blue edge from some vertex in $R'$ to some vertex in $R$.
\end{itemize}}

\smallskip

\noindent
First note that (a) follows immediately from Claim~2(b) and the definition of a red $(\delta^{1/4} n,\eps_1 n)$-clique.
To prove (b), suppose that some vertex $v\in V''\setminus R$ has both a red and a blue neighbour in $R$. Claim~2(a) and the fact that $|R|\ge \eps_1 n$ imply that either $v$ has at least $\eps_1 n/3$ red neighbours in $R$ or at least
$\eps_1 n/3$ blue neighbours in $R$ (or both). Suppose that the former holds (the argument for the latter is similar).
Let $u\in R$ be a blue neighbour of $v$. Since $R$ is a red $(\delta^{1/4} n,\eps_1 n)$-clique, all but at most $\delta^{1/4} n<\eps_1 n/3$ vertices of $R$ are red neighbours of $u$.
So there exists a red neighbour $u'\in R$ of $u$ which is also a red neighbour of $v$. Then the triangle $uu'v$ contradicts Claim~2(b).
This proves (b). The argument for (c) is similar. (d) follows from (b) and Claim~2(a),%
    \COMMENT{Suppose that $uu'$ is a red edge and $vv'$ is a blue edge with $u,v\in R$ and $u',v'\in R'$. Then (b) and Claim~2(a) together imply that
there exists $w'\in R'$ such that $w'$ is both a red neighbour of $u$ and a blue neighbour of $v$, a contradiction.}
while (e) follows from (b), (c) and Claim~2(a).

\vspace{0.2cm}
\noindent{\bf Claim 4:} \textit{There exists a collection $\cR$ of pairwise disjoint red $(\delta^{1/4} n,\eps_1 n)$-cliques such that, writing $V_\cR$
for the set of all those vertices in $V''$ covered by these red $(\delta^{1/4} n,\eps_1 n)$-cliques, the following holds:
\begin{itemize}
\item[{\rm (a)}] For every $R\in \cR$ and every $v\in R$ all red neighbours of $v$ lie in $R$.
\item[{\rm (b)}] Every $v\in V''\setminus V_\cR$ has less than $\eps_1 n$ red neighbours (and all of these lie in $ V''\setminus V_\cR$).
\end{itemize}}

\smallskip

\noindent
To prove Claim~4, let $\cR$ be a collection of pairwise disjoint red $(\delta^{1/4} n,\eps_1 n)$-cliques such that the set $V_{\cR}$
of all those vertices in $V''$ covered by these red $(\delta^{1/4} n,\eps_1 n)$-cliques is maximal and, subject to this condition, such that $|\cR|$ is minimal.
We will show that $\cR$ is as required in Claim~4.

To prove that Claim~4(a) holds, suppose first that there is some vertex $x\in V''\setminus V_{\cR}$ that has a red neighbour in some $R\in \cR$. Then Claims~2(a) and~3(b) together imply that $|(R\cup \{x\})\setminus N^\text{red}_{G''}(x)|\leq \delta^{1/4} n$. Moreover, by Claims~2(a) and~3(a),(b) we have that $|(R\cup \{x\})\setminus N^\text{red}_{G''}(v)|\leq \delta^{1/4} n$ for every $v\in R$. So $R\cup \{x\}$ is a red $(\delta^{1/4} n,\eps_1 n)$-clique, contradicting our choice of $\cR$.

Suppose next that there are distinct $R,R'\in \cR$ such that $G''$ contains a red edge between $R$ and $R'$. Then Claims~3(a),(d) imply that $G''[R\cup R']$ does not contain a blue edge. Together with Claim~2(a) this implies that $R\cup R'$ is a red $(\delta^{1/4} n,\eps_1 n)$-clique, again contradicting our choice of $\cR$. Altogether this proves Claim~4(a).

To check Claim~4(b), suppose that some $v\in V''\setminus V_{\cR}$ has at least $\eps_1 n$ red neighbours.
Claim~2(b) implies that $G''[N^\text{red}_{G''}(v)]$ cannot contain a blue edge. Together with Claim~2(a) this implies that
$G''[N^\text{red}_{G''}(v)]$ is a red $(\delta^{1/4} n,\eps_1 n)$-clique. But Claim~4(a) implies that $N^\text{red}_{G''}(v)\subseteq V''\setminus V_{\cR}$,
contradicting our choice of $\cR$. This completes the proof of Claim~4.

\medskip

Let $G'''$ be the semi-oriented graph obtained from $G''$ by deleting all the red edges which are not covered by
some $R\in\cR$. Note that by Claim~4(b) at most $\eps_1 n^2$ red edges are deleted.
Let $J$ be the oriented graph obtained from $G'''$ by contracting each $R\in \cR$ into a single vertex $v_R$.
So $V(J)$ consists of all these vertices $v_R$ as well as all the vertices in $V''\setminus V_{\cR}$.
Let $J_2:=J[V''\setminus V_{\cR}]=G'''[V''\setminus V_{\cR}]$ and let $J_1:=J-V(J_2)$. Claims~3(c),(e) and Claim~4(a) together imply that $J$ is indeed an oriented graph. 
Moreover, by Claim~2(a) $J_1$ is a tournament and $J[V(J_1),V(J_2)]$ is a bipartite tournament (i.e.~for all $v_R\in V(J_1)$ and $v\in V(J_2)$
either $v_Rv$ or $vv_R$ is a directed edge of~$J$).  

\vspace{0.2cm}
\noindent{\bf Claim 5:} \ 
\textit{
\begin{itemize}
\item[{\rm (a)}] $J$ does not contain a copy of $C_3$ having at least one vertex in $V(J_1)$.
\item[{\rm (b)}] $J_1$ is a transitive tournament.
\item[{\rm (c)}] $J$ can be made into a transitive tournament by changing at most $\eps_2 n^2$ edges in $E(J_2)$.
\end{itemize}}

\smallskip

\noindent
Suppose that (a) does not hold and let $xyv_R$ be a copy of $C_3$ in $J$. We only consider the case when $x\in V(J_1)$ and $y\in V(J_2)$; the
other cases are similar. So let $R'\in \cR$ be such that $x=v_{R'}$. Claim~2(a)%
    \COMMENT{need Claim~2(a) to ensure that $y$ is joined to most vertices in $R'$}
and the definition of $J$ together imply that $R'$ contains a blue
in-neighbour $x'$ of $y$ (in $G''$). 
Moreover, Claims~2(a) and~3(c),(e) imply that $|R\setminus N^-_{G''}(x')|\le \delta^{1/4} n$ and $|R\setminus N^+_{G''}(y)|\le \delta^{1/4} n$.
Together with the fact that $R$ is a red $(\delta^{1/4} n,\eps_1 n)$-clique this implies that $R$ contains a path $P=u\dots v$ of
length $k-2$ where $u\in N^+_{G''}(y)$ and $v\in N^-_{G''}(x')$. So $Px'y$ is a $C_{k+1}$ in $G''$ in which all the blue edges are oriented consistently. Using the fact that the edge $vx'$ is blue, it is now easy to see that
$Px'y$ corresponds to a $C_{k+1}$ in $G$, a contradiction. This proves (a). (b) follows from (a) and our previous observation that
$J_1$ is a tournament.%
   \COMMENT{Any tournament $T$ which contains a cycle must contain $C_3$. Indeed, if $x_1\dots x_r$ is a cycle, then $x_1x_3\in E(T)$
(else there is a $C_3$ in $T$). So $x_1x_4\in E(T)$. Continue to see that $x_1x_{r-1}\in E(T)$. So $x_1x_{r-1}x_r$ is a $C_3$.}

It remains to prove~(c).
Note that $e(J_2)\ge \binom{|J_2|}{2}-2\eps_1 n^2$ by Claim~2(a) and the definition of $J$ (and of $G'''$). Moreover, $J_2=G'''[V''\setminus V_{\cR}]$ is a $C_{k+1}$-free oriented graph.
So Lemma~\ref{cycle stability} implies that $J_2=T_{|J_2|}\pm \eps_2 n^2$. Let $\sigma_2: V(J_2)\to [|J_2|]$ be a transitive-optimal
ordering of the vertices of $J_2$. Let $r:=|J_1|=|\cR|$ and let $v_{R_1},\dots v_{R_r}$ be the unique transitive ordering of the vertices of $J_1$.
We claim that for every vertex $x\in V(J_2)$ there exists an index $i_x\in [r]$ such that all the $v_{R_i}$ with $i\le i_x$ are in-neighbours
of $x$ in $J$ while all the $v_{R_i}$ with $i> i_x$ are out-neighbours of $x$ in $J$. (Indeed, suppose not. Since
$J[V(J_1),V(J_2)]$ is a bipartite tournament this implies that there are indices $i<j$ such that $v_{R_i}$ is an out-neighbour of $x$ in $J$
and $v_{R_j}$ is an in-neighbour of $x$ in $J$. But then $xv_{R_i}v_{R_j}$ is a copy of $C_3$ contradicting~(a).)
For each $i\in [r]$ let $X_i:=\{x\in V(J_2): i_x=i\}$. Note that there are no indices $i<j$ such that
$J$ contains a directed edge from some vertex $x\in X_j$ to some vertex $x'\in X_i$ (otherwise $xx'v_{R_{i+1}}$ would be a copy of~$C_3$ contradicting~(a)).
Consider the vertex ordering $\sigma$ obtained from
$v_{R_1},\dots ,v_{R_r}$ by including all the vertices in $X_i$ between $v_{R_i}$ and $v_{R_{i+1}}$ in the
ordering induced by $\sigma_2$ (for each $i\in [r]$). This vertex ordering shows that~(c) holds.

\medskip

Recall that for each $R\in \cR$ the set $\bigcup R$ is a subset of $V(G)$ of size $k|R|$.

\vspace{0.2cm}
\noindent{\bf Claim 6:} \textit{Each red $(\delta^{1/4} n,\eps_1 n)$-clique $R\in \cR$ satisfies $G[\bigcup R]=DK_{|R|k/2,|R|k/2}\pm \delta^{1/5} n^2$.}

\smallskip

\noindent
To prove Claim~6, pick $v\in R$ and write $N^\text{red}_{G''}(v)\cap R=\{v_1,\dots,v_s\}$.
Recall that $v$ corresponds to a copy of $DK_{k/2,k/2}$ in $G$, and let $A$ and $B$ denote the vertex classes of this copy.
Similarly, each $v_i$ corresponds to a copy of $DK_{k/2,k/2}$ in $G$. Let $A_i$ and $B_i$ denote its vertex classes.
Recall from the definition of $G''$ that $G[A\cup A_i\cup B\cup B_i]=DK_{k,k}$. By swapping $A_i$ and $B_i$ if necessary, we may
assume that the vertex classes of this copy of $DK_{k,k}$ are $A\cup A_i$ and $B\cup B_i$. Since $G$ is $C_{k+1}$-free,
neither $G[A_1\cup \dots \cup A_s]$ nor $G[B_1\cup \dots \cup B_s]$ contains an edge. Thus whenever $v_iv_j$ is a red edge in $G''$ then
$G[A_i\cup A_j\cup B_i\cup B_j]$ is a copy of $DK_{k,k}$ with vertex classes $A_i\cup A_j$ and $B_i\cup B_j$.
But since $R$ is a red $(\delta^{1/4} n,\eps_1 n)$-clique, for each $i\in [s]$ all but at most $\delta^{1/4} n$ vertices in $\{v_1,\dots,v_s\}$
are red neighbours of $v_i$ and $|R\setminus \{v_1,\dots,v_s\}|\le \delta^{1/4} n$. Thus $G[\bigcup R]=DK_{|R|k/2,|R|k/2}\pm \delta^{1/5} n^2$,
as required.
  
\medskip
Using Claims~5(c) and~6 it is now straightforward to check that $G$ can be made into a transitive-bipartite blow up by
changing at most $\varepsilon n^2$ edges.
\endproof

We now have all the tools we need to show that almost all $C_k$-free oriented graphs are close to acyclic, and that for all even $k$ almost all $C_k$-free digraphs are close to acyclic, and that for all odd $k$ almost all $C_k$-free digraphs are close to a transitive-bipartite blow up. The proof of Lemma~\ref{quadratically many backwards} is almost identical to that of Lemma~\ref{step 0}, using Lemmas~\ref{cycle stability},~\ref{cycle stability digraphs} and~\ref{odd cycle stability} instead of Lemma~\ref{lem:sta}, and so is omitted here.

\begin{lemma}\label{quadratically many backwards}
For every $k\in \mathbb{N}$ with $k\geq 3$ and any $\a > 0$ there exists $\eps>0$ such that the following holds for all sufficiently large $n$.
\begin{enumerate}[{\rm (i)}]
\item All but at most $ f(n, C_k) 2^{-\eps n^2}$ $C_k$-free oriented graphs on $n$ vertices can be made into subgraphs of $T_n$ by changing at most $\a n^2$ edges.
\item If $k$ is even then all but at most $ f^*(n, C_k) 2^{-\eps n^2}$ $C_k$-free digraphs on $n$ vertices can be made into subgraphs of $T_n$ by changing at most $\a n^2$ edges.
\item If $k$ is odd then all but at most $ f^*(n, C_k) 2^{-\eps n^2}$ $C_k$-free digraphs on $n$ vertices can be made into a subgraph of a transitive-bipartite blow up by changing at most $\a n^2$ edges.
\end{enumerate}
\end{lemma}

\section{Typical $C_k$-free oriented graphs and digraphs are not acyclic}\label{Section: Typical C_k-free oriented graphs and digraphs are not transitive}

Let $\mathcal{O}_{n,k}$ be the set of all labelled $C_k$-free oriented graphs on $n$ vertices and let $\mathcal{O}^*_{n,k}$ be the set of all labelled $C_k$-free digraphs on $n$ vertices. We show that almost all graphs in $\mathcal{O}_{n,k}$ and almost all graphs in $\mathcal{O}^*_{n,k}$ have at least $c n/\log n$ backwards edges in a transitive-optimal ordering, for some constant $c>0$. Let $\mathcal{O}_{n,k,r}$ be the set of all labelled $C_k$-free oriented graphs on $n$ vertices with exactly $r$ backwards edges in a transitive-optimal ordering. Let $\mathcal{O}_{n,k,\leq r}:=\bigcup_{i\in \{0,1,\dots, \lfloor r\rfloor \}} \mathcal{O}_{n,k,i}$, and define the digraph analogues $\mathcal{O}^*_{n,k,r}$ and $\mathcal{O}^*_{n,k,\leq r}$ in a similar way.

\begin{lemma}\label{linearly many backwards}
Let $k\geq 3$ and let $n\in \mathbb{N}$ be sufficiently large. Then
\begin{enumerate}[{\rm (i)}]
\item $|\mathcal{O}_{n,k, n/2^{13} }|\geq 2^{{n}/{2^{14}}} |\mathcal{O}_{n,k,\leq n/\left( 2^{14}\log n \right)}|$,
\item $|\mathcal{O}^*_{n,k, n/2^{13}}|\geq 2^{{n}/{2^{14}}} |\mathcal{O}^*_{n,k,\leq n/\left( 2^{14}\log n \right)}|$.
\end{enumerate}
\end{lemma}

Note that Lemma~\ref{linearly many backwards} together with Lemma~\ref{quadratically many backwards} immediately yields Theorem~\ref{C_k free main theorem}.

\removelastskip\penalty55\medskip\noindent{\bf Proof of Lemma~\ref{linearly many backwards}.}
We only prove the case $k=3$ of {\rm (i)} here; the proofs for {\rm (ii)} and the case $k>3$ are very similar\COMMENT{For $k>3$, we use flippable $3$-sets instead of flippable $4$-sets. The proof is almost identical, just using different numerical values - see further COMMENTS.}. Let $m_2:=\lfloor n/2^{13} \rfloor$. Fix $m_1\in \mathbb{Z}$ with $0\leq m_1\leq m_2/(2\log n)$. For every oriented graph $G$ fix some transitive-optimal ordering $\sigma_G: V(G) \to [n]$.

Consider an auxiliary bipartite graph $H$ with vertex classes $\mathcal{O}_{n,3,m_1}$ and $\mathcal{O}_{n,3,0}$ whose edge set is defined as follows. Let there be an edge in $H$ between $A\in \mathcal{O}_{n,3,m_1}$ and $B\in \mathcal{O}_{n,3,0}$ if the graph $B$ can be obtained from the graph $A$ by deleting the $m_1$ backwards edges with respect to $\sigma_A$. Note that every graph generated in this way from a graph $A\in \mathcal{O}_{n,3,m_1}$ belongs to $\mathcal{O}_{n,3,0}$, so $A$ certainly has at least one neighbour in $\mathcal{O}_{n,3,0}$.

We claim that, in $H$, a graph $B\in \mathcal{O}_{n,3,0}$ has at most $\binom{n^2/2}{m_1}2^{m_1}$ neighbours in $\mathcal{O}_{n,3,m_1}$. Indeed, any graph in $\mathcal{O}_{n,3,m_1}$ that can generate $B$ in the described way can be obtained from $B$ by choosing exactly $m_1$ of the at most $n^2/2$ pairs of vertices that have no edge between them in $B$, and then adding edges between them with some orientations\COMMENT{Note that not every choice of orientations here will give a graph in $\mathcal{O}_{n,3,m_1}$, but since every graph in $\mathcal{O}_{n,3,m_1}$ can be obtained in this way, an upper bound on the number of ways of orienting these $m_1$ edges is an upper bound on the number of graphs in $\mathcal{O}_{n,3,m_1}$ that have some edge in each of the $m_1$ locations we have chosen, and this is all we require.} (for which there are $2^{m_1}$ possibilities).

Together with our previous observation that, in $H$, every graph $A\in \mathcal{O}_{n,3,m_1}$ has at least one neighbour in $\mathcal{O}_{n,3,0}$, this implies that
\begin{equation}\label{penultimate equation}
|\mathcal{O}_{n,3,m_1}|\leq \sum\limits_{A\in \mathcal{O}_{n,3,m_1}}d_{\mathcal{O}_{n,3,0}}(A)=\sum\limits_{B\in \mathcal{O}_{n,3,0}}d_{\mathcal{O}_{n,3,m_1}}(B)\leq |\mathcal{O}_{n,3,0}|\binom{n^2/2}{m_1}2^{m_1}.
\end{equation}

For a graph $G\in \mathcal{O}_{n,3,0}$ we define a \textit{flippable $4$-set} in $G$ to be any set of $4$ vertices, with labels $w,x,y,z$ say, satisfying the following:
\begin{itemize}
\item the vertices $w,x,y,z$ are consecutive in the ordering $\sigma_G$; that is $\sigma_G(w)+3=\sigma_G(x)+2= \sigma_G(y)+1=\sigma_G(z)$,
\item $\sigma_G(w)-1$ is divisible by $4$.
\end{itemize}
Note that every graph in $\mathcal{O}_{n,3,0}$ has $\lfloor n/4 \rfloor$ flippable $4$-sets\COMMENT{$\lfloor n/3 \rfloor$ flippable $3$-sets when $k>3$}.

Now consider an auxiliary bipartite graph $H'$ with vertex classes $\mathcal{O}_{n,3,0}$ and $\mathcal{O}_{n,3,m_2}$ whose edge set is defined as follows. Let there be an edge in $H'$ between $B\in \mathcal{O}_{n,3,0}$ and $C\in \mathcal{O}_{n,3,m_2}$ if the graph $C$ can be obtained from the graph $B$ by choosing exactly $m_2$ flippable $4$-sets in $B$ with respect to $\sigma_B$ and, for each flippable $4$-set $w,x,y,z$ chosen, deleting all edges between the vertices $w,x,y,z$ and then adding the edges of a $4$-cycle $wxyz$. Note that every graph generated in this way from a graph $B\in \mathcal{O}_{n,3,0}$ belongs to $\mathcal{O}_{n,3,m_2}$.

We claim that, in $H'$, a graph $B\in \mathcal{O}_{n,3,0}$ has exactly $\binom{\lfloor n/4 \rfloor}{m_2}$ neighbours\COMMENT{$\binom{\lfloor n/3 \rfloor}{m_2}$ neighbours when $k>3$} in $\mathcal{O}_{n,3,m_2}$. Indeed, the neighbours of $B$ are precisely those graphs generated by choosing exactly $m_2$ of the exactly $\lfloor n/4 \rfloor$ flippable $4$-sets in $B$ with respect to $\sigma_B$, and then changing the edges between pairs of vertices in these flippable $4$-sets in the described way. Each choice of $m_2$ flippable $4$-sets generates a different graph. So the claim holds.

We claim also that, in $H'$, a graph $C\in \mathcal{O}_{n,3,m_2}$ has at most $2^{8m_2}$ neighbours\COMMENT{$2^{5m_2}\leq 2^{8m_2}$ neighbours when $k>3$} in $\mathcal{O}_{n,3,0}$. Indeed, first note that any graph in $\mathcal{O}_{n,3,m_2}$ with at least one neighbour in $\mathcal{O}_{n,3,0}$ contains exactly $m_2$ induced $4$-cycles. Any graph in $\mathcal{O}_{n,3,0}$ that can generate $C$ in the described way can be obtained from $C$ by choosing for each of the $m_2$ induced $4$-cycles an ordering of the $4$ vertices respecting the order of the $4$-cycle (of which there are $4$), and then changing the edges between pairs of vertices in these $4$-cycles to some transitive configuration with respect to the chosen ordering (for which there are $2^6$ possibilities\COMMENT{$2^3$ transitive configurations of $3$-cycles when $k>3$}). So indeed the claim holds.

So using these degree bounds gives us that
\begin{equation}\label{last equation}
|\mathcal{O}_{n,3,0}| \binom{\lfloor n/4 \rfloor}{m_2}= \sum\limits_{B\in \mathcal{O}_{n,3,0}}d_{\mathcal{O}_{n,3,m_2}}(B)=\sum\limits_{C\in \mathcal{O}_{n,3,m_2}}d_{\mathcal{O}_{n,3,0}}(C)\leq |\mathcal{O}_{n,3,m_2}|2^{8m_2}.
\end{equation}

Now (\ref{penultimate equation}) and (\ref{last equation}) together imply that
\begin{equation}\label{really last equation}
\frac{|\mathcal{O}_{n,3,m_2}|}{|\mathcal{O}_{n,3,m_1}|}\geq \frac{\binom{\lfloor n/4 \rfloor}{m_2}}{\binom{n^2/2}{m_1}2^{m_1}2^{8m_2}}.
\end{equation}
Since $n$ is sufficiently large we have that $\binom{\lfloor n/4 \rfloor}{m_2}\geq \left( \frac{n}{8m_2} \right)^{m_2}$ and $\binom{n^2/2}{m_1}2^{m_1}\leq n^{2m_1}$. Hence the right hand side of (\ref{really last equation}) is at least
\begin{equation*}
\left( \frac{n}{2^{11}m_2} \right)^{m_2} n^{-2m_1}\geq \left( \frac{n}{2^{11}m_2} \right)^{m_2} n^{-\frac{m_2}{\log n}} = 2^{m_2 \log \left( \frac{n}{2^{11}m_2} \right) - \frac{m_2}{\log n} \log n} \geq 2^{m_2}.
\end{equation*}
So this together with (\ref{really last equation}) gives us that $|\mathcal{O}_{n,3,m_2}|\geq 2^{m_2} |\mathcal{O}_{n,3,m_1}|$ for any integer $0\leq m_1\leq m_2/(2 \log n)$. So since $n$ is sufficiently large,
$$|\mathcal{O}_{n,3,m_2}|\geq \frac{2^{m_2}}{n}|\mathcal{O}_{n,3,\leq m_2/(2 \log n)}|\geq 2^{{n}/{2^{14}}}|\mathcal{O}_{n,3,\leq m_2/(2 \log n)}|,$$
as required.
\endproof

\section{Acknowledgment}
This research was carried out whilst Yi Zhao was visiting the School of Mathematics at the University of Birmingham, UK. He would like to thank the school for the hospitality he received.

\medskip
{\footnotesize \obeylines \parindent=0pt

Daniela K\"uhn, Deryk Osthus, Timothy Townsend
School of Mathematics
University of Birmingham
Edgbaston
Birmingham
B15 2TT
UK

\begin{flushleft}
{\it{E-mail addresses}:
\tt{\{d.kuhn, d.osthus, txt238\}}@bham.ac.uk}
\end{flushleft}

\

Yi Zhao
Department of Mathematics \& Statistics
Georgia State University
30 Pryor Street
Atlanta
GA 30303
USA

\begin{flushleft}
{\it{E-mail address}:
\tt{yzhao6}@gsu.edu}
\end{flushleft}
}

\end{document}